\documentclass[twoside,11pt]{article}

\usepackage{jmlr2e}

\usepackage{color}
\usepackage{epstopdf}
\usepackage{amsmath,amstext,amsfonts,amssymb,bbm,float}
\usepackage[ruled,vlined]{algorithm2e}
\usepackage{enumerate}
\usepackage{caption}
\usepackage{subcaption}
\usepackage{graphicx}
\usepackage{bbm}
\usepackage{float}
\usepackage{natbib}

\ShortHeadings{Weighted Singular Value Thresholding}{Dutta et al.}
\firstpageno{1}

\begin{document}

\title{Weighted Singular Value Thresholding\\ and its Application to Background Estimation}

\author{\name Aritra Dutta \email d.aritra2010@knights.ucf.edu \\
       \addr Department of Mathematics\\
       University of Central Florida\\
       Orlando, FL 32816, USA
       \AND
       \name Boqing Gong \email  bgong@crcv.ucf.edu\\
       \addr Center for Research in Computer Vision\\
       University of Central Florida\\
       Orlando, FL 32816-2365, USA
       \AND
       \name Xin Li  \email xin.li@ucf.edu\\
       \addr Department of Mathematics\\
       University of Central Florida\\
       Orlando, FL 32816, USA
       \AND 
       \name Mubarak Shah \email shah@crcv.ucf.edu\\
       \addr Center for Research in Computer Vision\\
       University of Central Florida\\
       Orlando, FL 32816-2365, USA}

\maketitle

\begin{abstract}
Singular value thresholding~(SVT) plays an important role in the well-known robust principal component analysis~(RPCA) algorithms which have many applications in computer vision and recommendation systems.~In this paper, we formulate and study a {\it weighted} singular value thresholding~(WSVT) problem, which uses a combination of the nuclear norm and a weighted Frobenius norm.~We present an algorithm to numerically solve WSVT and establish the convergence of the algorithm.~As a proof of concept, we apply WSVT with a simple choice of weight learned from the data to the background estimation problem in computer vision.~The numerical experiments show that our method can outperform RPCA algorithms.~This indicates that instead of tackling the computationally expensive $\ell_1$ norm employed in RPCA, one may switch to the weighted Frobenius norm and achieve about the same or even better performance. 
\end{abstract}

\section{Introduction}

The classical principal component analysis~(PCA)~problem~\citep{svd,pca} can be defined as a best approximation to a given matrix $X\in R^{m\times n}$ by a rank $r$ matrix under the Frobenius norm as follows:
\begin{eqnarray}\label{pca}
\hat{B}=\displaystyle{\arg\min_{\substack{B\\{\rm r}(B)\leq r}}\|X-B\|_F},
\end{eqnarray}
\noindent where ${\rm r}(B)$ denotes the rank of the matrix $B$. If $U\Sigma V^T$ is a singular value decomposition~(SVD) of $X$, then solutions to this problem are given by thresholding on the singular values of $X
$: $\hat{B}=U{\mathbf H}_r(\Sigma)V^T,$
where ${\mathbf H}_r$ is the hard-thresholding operator that keeps the $r$ largest singular values and replaces the others by $0$.

As it turns out, the nuclear norm $\|B\|_*$, the sum of the singular values of $B$, is a good substitution for ${\rm r}(B)$ in many problems~\citep{candesplan}.~\citet{caicandesshen}, used this to propose the following unconstrained convex optimization formulation of a low rank approximation problem:
\begin{equation}\label{svt}
\hat{B}=\arg\min_B\{\frac{1}{2}\|X-B\|_F^2+\tau \|B\|_*\}.
\end{equation}
The solution to this problem can be explicitly given using the SVDs of $X=U\Sigma V^T$ by
$
\hat{B}=U{\mathbf S}_{\tau}(\Sigma)V^T,
$
where ${\mathbf S}_{\tau}(\cdot)$ is the element-wise application of the soft-thresholding operator defined as $ {\mathbf S}_{\tau}(x)={\rm sign}(x)(|x|-\tau)_+.$~\citet{caicandesshen} referred this as the {\it singular value thresholding} (SVT) method.

It is a well-known fact that the solution to the classical PCA problem is numerically sensitive to the presence of outliers in the data matrix~\citep{LinChenMa,APG,candeslimawright}.~In other words, if the matrix $X$ is perturbed by one large value, 
the explicit formula for its low rank approximation would yield a much different solution than the un-perturbed one.~On the other hand,~$\ell_1$~norm does encourage sparsity when the norm is made small.~To solve the problem of separating the sparse outliers added to a low-rank matrix,~\citet{candeslimawright}~replaced the Frobenius norm in the SVT problem by the $\ell_1$~norm~and introduced the {\it Robust PCA} (RPCA) method~(see also~\citet{LinChenMa}):
\begin{equation}\label{rpca}
	\min_B\{\|X-B\|_{\ell_1}+\lambda \|B\|_*\}.
\end{equation}
Unlike in the classical PCA and SVT problems, the RPCA problem has no closed form solution.~Various numerical procedures have been proposed to solve it.~\citet{LinChenMa}, 
proposed two iterative methods:~exact and inexact Augmented Lagrange Method (EALM and iEALM).~The iEALM method turns out to be equivalent to the alternating direction method (ADM) later proposed by \citet{tao-yuan}.~\citet{APG} proposed the accelerated proximal gradient~(APG) method to solve the RPCA problems numerically as well.~Note that in all the numerical procedures for solving RPCA, the solution to the SVT problem is used as an important auxiliary step. 

In this paper, we propose an alternative solution to the sensitivity of PCAs to the outliers by simply introducing a weight matrix in the Frobenius norm~in~(\ref{svt}).~In particular,~using a non-singular weight matrix $W\in \mathbb{R}^{n\times n}$, we focus on the following weighted singular value thresholding~(WSVT) problem:
\begin{eqnarray}\label{weighted svt}
\min_B\{\frac{1}{2}\|(X-B)W\|_F^2+\tau \|B\|_*\},
\end{eqnarray}
where the weight matrix $W$ is user provided or automatically inferred from the data.~Our experiments suggest, a properly inferred weight matrix $W$ may eliminate the effect of the outliers in the data $X$, sharing the similar spirit as the $\ell_1$ norm employed in RPCA.~We develop a numerical algorithm to solve WSVT, present the analysis about its convergence, and also conduct some experiments to compare it with RPCA.

We note that using weighted Frobenius norm is not new in low rank matrix approximation problems.~In 2003,~\citet{srebro},~studied the following weighted low-rank approximation problem:~for a given matrix $X\in\mathbb{R}^{m\times n}$ and $r\le\min\{m,n\}$ find
\begin{align}\label{srebro_jakola}
\min_{\substack{{B}\in\mathbb{R}^{m\times n}\\{\rm r}({B})\le r}}\|({X}-B)\odot \tilde{W}\|_{F}^2,
\end{align}
where $\tilde{W}\in\mathbb{R}^{m\times n}$ is a non-negetive weight matrix and $\odot$ denotes the element-wise matrix multiplication. They pointed out that, in general, there is no closed form solution to~(\ref{srebro_jakola}).~At about the same time,~\citet{manton} proposed a problem with a generalized norm:
\begin{align}\label{manton}
\min_{\substack{{B}\in\mathbb{R}^{m\times n}\\{\rm r}({B})\le r}}\|X-B\|_{Q}^2,
\end{align}
where $Q \in \mathbb{R}^{mn\times mn}$ is a symmetric positive definite weight matrix,~$\|X-B\|_Q^2:={\rm vec}(X-B)^TQ{\rm vec}(X-B)$, and ${\rm vec}(\cdot)$ is an operator which maps the entries of $\mathbb{R}^{m\times n}$ to vectors in $\mathbb{R}^{mn\times 1}$ by stacking the columns.


The weighted Frobenius norm used in part of the WSVT problem~(\ref{weighted svt}) is a special case of both~(\ref{srebro_jakola})~and~(\ref{manton}). Note that both~(\ref{srebro_jakola})~and~(\ref{manton}) are constrained problems while~(\ref{weighted svt}) is an unconstrained problem. Due to the special structure of the WSVT problem, the numerical procedures for solving~(\ref{srebro_jakola})~and~(\ref{manton}) proposed in the literature~\citep{srebro, wibergjapan, wiberg, srebromaxmatrix, Buchanan, erikssonhengel, markovosky,markovosky1,markovosky3,markovosky4} cannot be directly applied to solving it. Moreover, we believe WSVT is worth studying as a standalone problem for the following reasons. One is that it serves as a natural alternative to RPCAs in many applications and is computationally inexpensive. The other is that the special structure of the objective function in (\ref{weighted svt}) allows us to present a detailed convergence analysis of the numerical algorithm which is usually hard to obtain in the algorithms for solving (\ref{srebro_jakola})~and~(\ref{manton})~\citep{srebro,manton,markovosky,wiberg,wibergjapan}.

To this end, one of our contributions in this paper is a numerical algorithm to solve~WSVT with theoretical convergence analysis.~An extended study of~(\ref{weighted svt}), when the nuclear norm is replaced by ${\rm r}(B)$ and the regular matrix multiplication with the weight matrix is replaced by more general pointwise multiplication, can be found in~\citet{duttali}.

The low rank approximation technique has been used to background estimation from the video sequences.~In 1999,~\citet{oliver} proposed that if the camera motion is presumably static then the background image sequence can be modeled as a low-dimensional linear subspace.~Therefore, the foreground layer which is relatively sparse comparing to the slowly changing background layer can be modeled as a sparse ``outlier'' component of the video sequence. In short, if each frame of the sequence is vectorized and arranged as columns of $X$, then $B$, the low-rank matrix is assumed to capture the background information considering $X-B$ sufficiently sparse.~In the past decade, a key application of RPCA problems is in background estimation from video sequences.~Thus a solution to~(\ref{rpca}) would give a reasonable estimation of the background frames from a video sequence.~On the other hand,~the SVT in~(\ref{svt}) fails to provide a comparable background estimation~(see also Section 4 for some numerical results).~We use this experimental setup to test our algorithm as well as the algorithms for RPCA.~For a thorough review of the most recent and traditional algorithms for solving background estimation problem, we refer the reader to~\cite{Bouwmans201431},~\citet{Sobral20144}, and~\citet{Bouwmans2016}.

The organization of the rest of the paper is as follows.~In Section 2, we propose a numerical solution to our WSVT problem for a general non-singular weight matrix $W$.~In Section 3, we present a detailed convergence analysis of our proposed numerical algorithm. In Section 4, using WSVT we propose a robust background estimation model and compare its performance with the RPCA algorithms.


\section{Solving the WSVT problem}

We propose a numerical algorithm to solve the WSVT problem~(\ref{weighted svt}) when $W$ is non-singular. The novelty of our WSVT algorithm is that by using auxiliary variables, we can employ the simple and fast alternating direction method~(ADM). Since $W$ is non-singular, we re-write~(\ref{weighted svt}) as:
\begin{eqnarray*}
	\min_C\{\frac{1}{2}\|XW-C\|_F^2+\tau \|CW^{-1}\|_*\}.
\end{eqnarray*}
Introduce a new variable $D$ with equality constraints as follows: $D=CW^{-1}$.~Then above problem becomes
\begin{eqnarray}
\min_{C,D\atop{D=CW^{-1}}}\{\frac{1}{2}\|XW-C\|_F^2+\tau \|D\|_*\}.
\end{eqnarray}
Next, we use the augmented Lagrange multiplier~(ALM) method~\citep{LinChenMa, sboyd} to solve this minimization problem.~Let
\begin{eqnarray*}
	L(C,D,Y,\mu)&=&\frac{1}{2}\|XW-C\|_F^2+\tau \|D\|_*+\langle Y,D-CW^{-1}\rangle+\frac{\mu}{2}\|D-CW^{-1}\|^2_F
\end{eqnarray*}
be the augmented Lagrangian function where $Y\in R^{m\times n}$ is the Lagrange multiplier and $\mu>0$, a balancing parameter.~To find a numerical solution to the minimization problem $\displaystyle{\min_{C,D} L(C,D,Y,\mu)}$, we use the alternating direction method via the following iterative updating scheme: \begin{eqnarray*}
	& C_{k+1}=\displaystyle{\arg\min_C L(C,D_k,Y_k,\mu_k)},\\
	& D_{k+1}=\displaystyle{\arg\min_D L(C_{k+1},D,Y_k,\mu_k)}.
\end{eqnarray*}
Note that, by completing the squares and keeping only the relevant terms in the augmented Lagrangian, we have
 \begin{eqnarray*}
 	& \displaystyle{\arg\min_C L(C,D_k,Y_k,\mu_k)}=\displaystyle{\arg\min_C\{\frac{1}{2}\|XW-C\|_F^2+\frac{\mu_k}{2}\|D_k-CW^{-1}+\frac{1}{\mu_k}Y_k\|_F^2\}},\\
 	& \displaystyle{\arg\min_D L(C_{k+1},D,Y_k,\mu_k)}=\displaystyle{\arg\min_D\{\tau\|D\|_*+\frac{\mu_k}{2}\|D-C_{k+1}W^{-1}+\frac{1}{\mu_k}Y_k\|_F^2\}}.
 \end{eqnarray*}
The solution to the first-subproblem can be derived by setting the gradient with respect to $C$ to $0$:
\begin{align*}
C_{k+1} = (XW + \mu_k{D_k} (W^{-1})^T+Y_k(W^{-1})^T)(I + \mu_k(W^TW)^{-1})^{-1};
\end{align*}
and, the second-subproblem is a SVT problem~(\ref{svt}) with $B=C_{k+1}W^{-1}-\frac{1}{\mu_k}Y_k$ and so its solution is
\begin{equation*}
D_{k+1}=U_kS_{\tau/\mu_k}(\Sigma_k)V_k^T,
\end{equation*}
where $U_k\Sigma_kV_k^T$ is a SVD of $(C_{k+1}W^{-1}-\frac{1}{\mu_k}Y_k)$.
We update $Y_k$ and $\mu_k$ by 
\begin{equation*}
Y_{k+1}=Y_k+\mu_k (D_{k+1}-C_{k+1}W^{-1}); ~\mu_{k+1}=\rho\mu_k,
\end{equation*}
for a fixed $\rho>1$. Algorithm~\ref{alg_1} presents the full numerical procedure.
\begin{algorithm}
	\SetAlgoLined
	\SetKwInOut{Input}{Input}
	\SetKwInOut{Output}{Output}
     \SetKwInOut{Init}{Initialize}
	\nl\Input{Data matrix $X \in \mathbb{R}^{m \times n}$, weight matrix $W \in \mathbb{R}_{+}^{n \times n}$ and $\tau >0, \rho>1$\;}
     \nl\Init {$C = XW,D = X,Y = 0,\mu >0$\;}
	\BlankLine
	\nl \While{not converged}
	{
		\nl $C = (XW + \mu{D} (W^T)^{-1}+Y(W^{-1})^T)(I_n + \mu(W^TW)^{-1})^{-1}$\;
		\nl $[U\;\;\Sigma\;\;V] = SVD(CW^{-1} -\frac{1}{\mu}Y)$\;
		\nl $D = US_{\frac{\tau}{\mu}}(\Sigma)V^T$\;
		\nl $Y = Y+\mu(D-CW^{-1})$\;
		\nl $\mu = \rho\mu$\;
	}
	\BlankLine
	\nl \Output{$B= CW^{-1}$}
	\caption{WSVT algorithm} \label{alg_1}
\end{algorithm}

\section{Convergence analysis}

{I}{n} this section we establish the convergence of our algorithm by following the main ideas of~\citet{LinChenMa} and~\citet{LiOreifej}.~Recall that $Y_{k+1}= Y_k+\mu_k(D_{k+1}-C_{k+1}W^{-1})$ and define ${\hat{Y}_{k+1}}:=Y_k+\mu_k(D_{k}-C_{k+1}W^{-1})$. We first state our main results. 
\begin{theorem}\label{theorem_3_3}
	Withe the notations introduced above, the following hold. 
	\begin{description}
		\item{(i)} The sequences $\{C_k\}$ and $\{D_k\}$ are convergent. Moreover,  $$
		\|D_k-C_kW^{-1}\|\le \frac{C}{\mu_k},~~k = 1,2,\cdots,$$ for some constant $C$ independent of $k$.
		\item{(ii)} If $L_{k+1} := L(C_{k+1},D_{k+1},Y_k,\mu_k)$, then the sequence $\{L_k\}$ is bounded above and 
		\begin{align*}
		& L_{k+1}-L_k \le \frac{\mu_k +\mu_{k-1}}{2}\|D_k-C_kW^{-1}\|_F^2=O(\frac{1}{\mu_k}),
		\;\;\text{for}\;\;k = 1,2,\cdots.
		\end{align*}
	\end{description}
\end{theorem}

\begin{theorem}\label{theorem_3_4}
	Let $(C_{\infty},D_{\infty})$ be the limit point of $(C_k,D_k)$ and define $$f_{\infty} =  \frac{1}{2}\|XW-C_{\infty}\|_F^2 +\tau\|D_{\infty}\|_\ast.$$ Then $C_{\infty}=D_{\infty}W$ and
	$$-O(\mu_{k-1}^{-2})\le \frac{1}{2}\|XW-C_k\|_F^2 +\tau\|D_k\|_\ast-f_{\infty} \le O(\mu_{k-1}^{-1}).$$
\end{theorem}

We provide the proofs of Theorems~\ref{theorem_3_3} and \ref{theorem_3_4} in Appendix A. 
\section{Experimental results}

In this section, as a proof of concept, we demonstrate the performance of our WSVT algorithm on two computer vision applications: background estimation from video sequences and shadow removal from face images under varying illumination conditions.~We show that, with a diagonal weight matrix $W$, we can improve the performance or achieve similar results as compared with other state-of-the-art unweighted low-rank algorithms, especially RPCAs.

\subsection{Background estimation}
Background estimation from video sequences is a classic computer vision problem.~One can consider the scene in the background is presumably static; thus, the background component is expected to be the low-rank part of the matrix $X$ that concatenates the video frames. Minimizing the rank of the matrix $X$ emphasizes the structure of the linear subspace containing the column space of the background.~However, the exact desired rank is often tuned empirically, as the ideal rank-one background is often unrealistic. 

In our experiments, we use three different sequences:~(i)~the Stuttgart synthetic video data set~\citep{cvpr11brutzer}, (ii)~the airport hall sequence, and~(iii)~the fountain sequence~from the PTIS dataset~\citep{PTIS}. We give qualitative analysis results on all three sequences. For performing the quantitative analysis between different methods, we use the Stuttgart video sequence. It is a computer generated sequence from the vantage point of a static camera located on the side of a building viewing a city intersection. The reason for choosing this sequence is two-fold.~First, this is a challenging video sequence which comprises both static and dynamic foreground objects and varying illumination in the background. Second, because of the availability of ample amount of ground truth, we can provide a rigorous quantitative comparison of the various methods. We choose the first 600 frames of the {\it Basic} scenario to capture the changing illumination and foreground object. Correspondingly, we have 600 ground truth frames.~The frames and ground truths are resized to  $64\times80$ and each is vectorized to a column vector of size $5120\times1$. Denote by the matrix as the concatenation of all the video frames, $X=\{vec(I_1),vec(I_2),\cdots,vec(I_{600})\}$, where $vec(I_i)\in\mathbb{R}^{5120\times 1}$ and $I_i\in\mathbb{R}^{64\times 80}$. Figure~\ref{video_frame} shows a sample video frame of the {\it Basic} scenario and the corresponding ground truth mask from the Stuttgart video sequence and demonstrates an outline of processing the video frames defined above. 
\begin{figure}
	\begin{minipage}{0.4\textwidth}
		\centering
		\begin{subfigure}{\textwidth}
			\includegraphics[width=\textwidth, height = 1.2in]{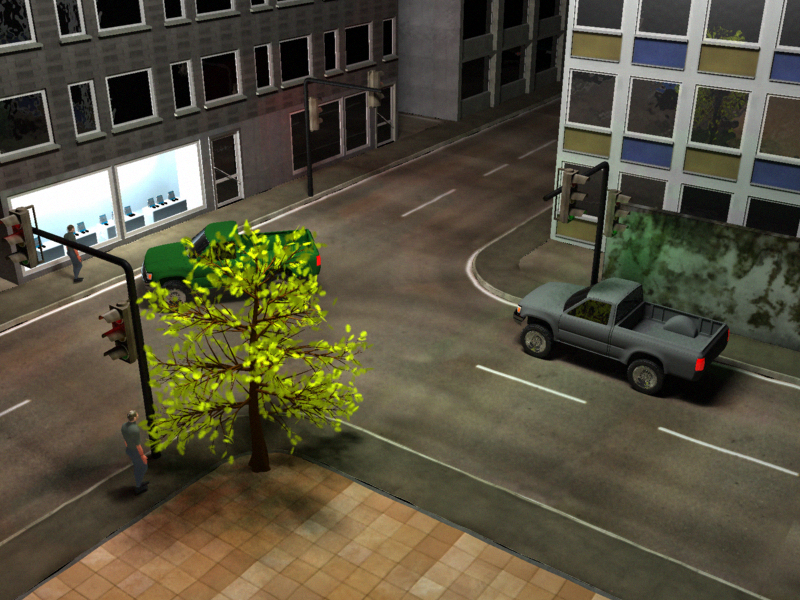}
			\caption{}\label{fig:leftA}
		\end{subfigure}
		\begin{subfigure}{\textwidth}
			\includegraphics[width=\textwidth,height = 1.2in]{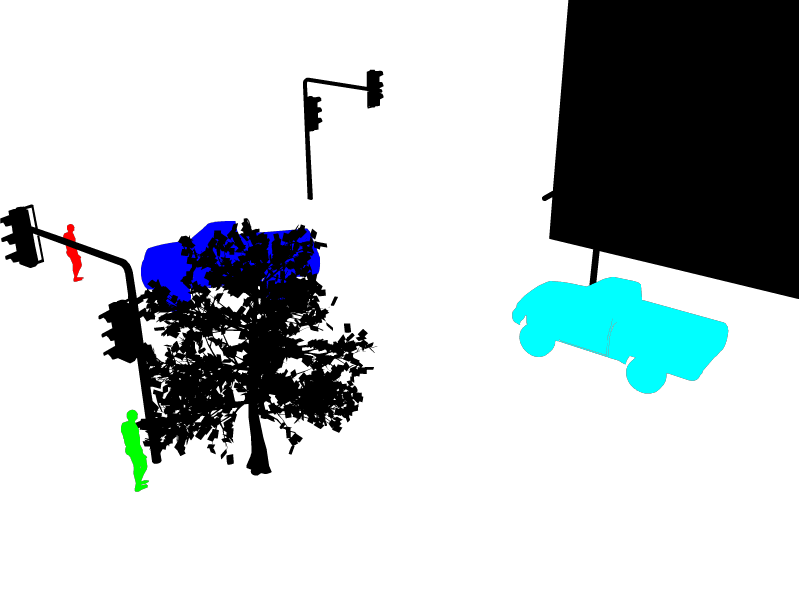}
			\caption{}\label{fig:leftB}
		\end{subfigure}
	\end{minipage}\hfill
	\begin{minipage}{0.55\textwidth}
		\centering
		\begin{subfigure}{\textwidth}
			\includegraphics[width=1.2\textwidth,height =2.5in]{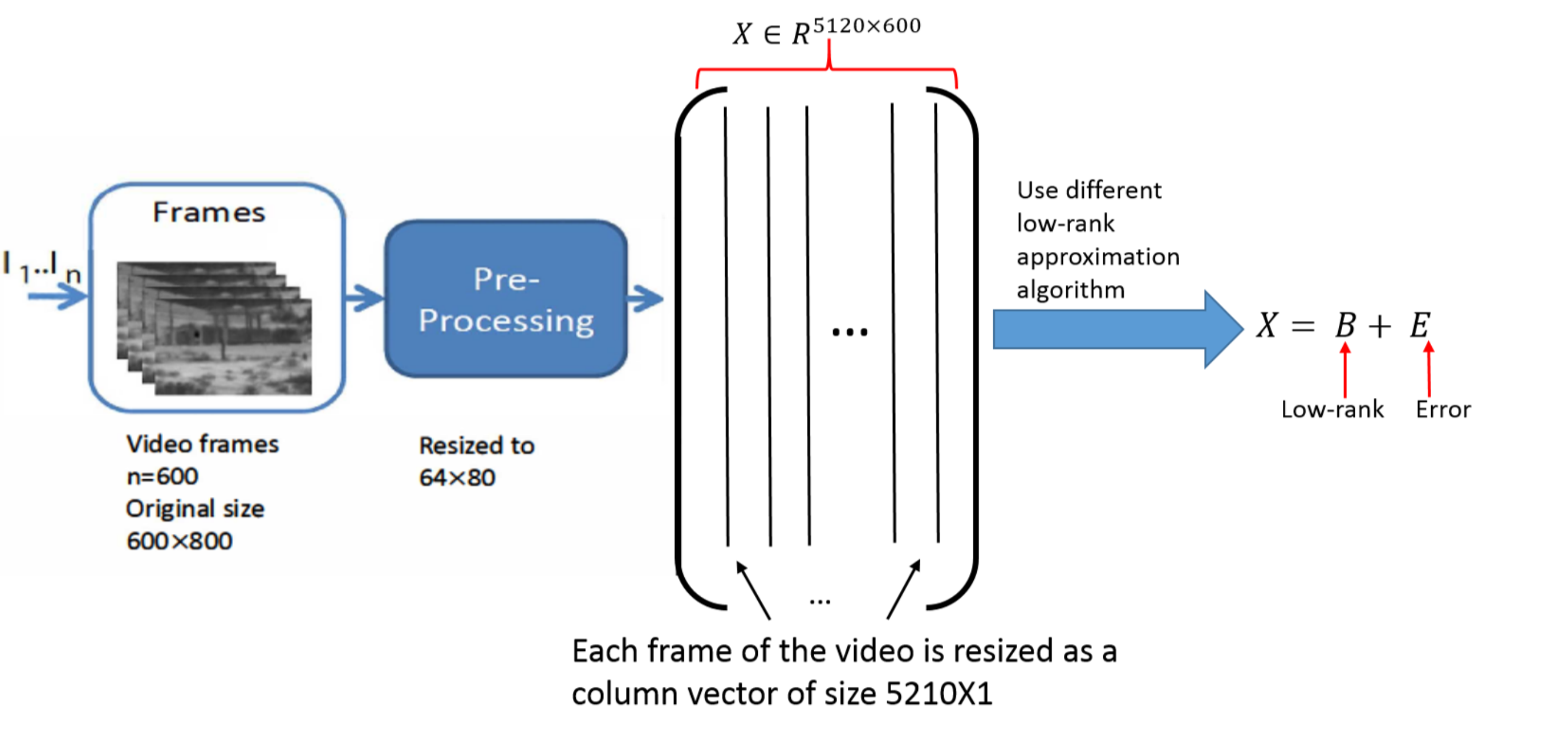}
			\caption{}\label{fig:rightA}
		\end{subfigure}
	\end{minipage}
	\caption{~(a)~Sample frame from the Stuttgart artificial video sequence and (b) the corresponding ground truth mask.~(c)~The framework for background estimation.}\label{video_frame}
\end{figure}
\begin{figure}[H]
	\centering
	\begin{subfigure}[b]{0.5\textwidth}
		\includegraphics[width=\textwidth]{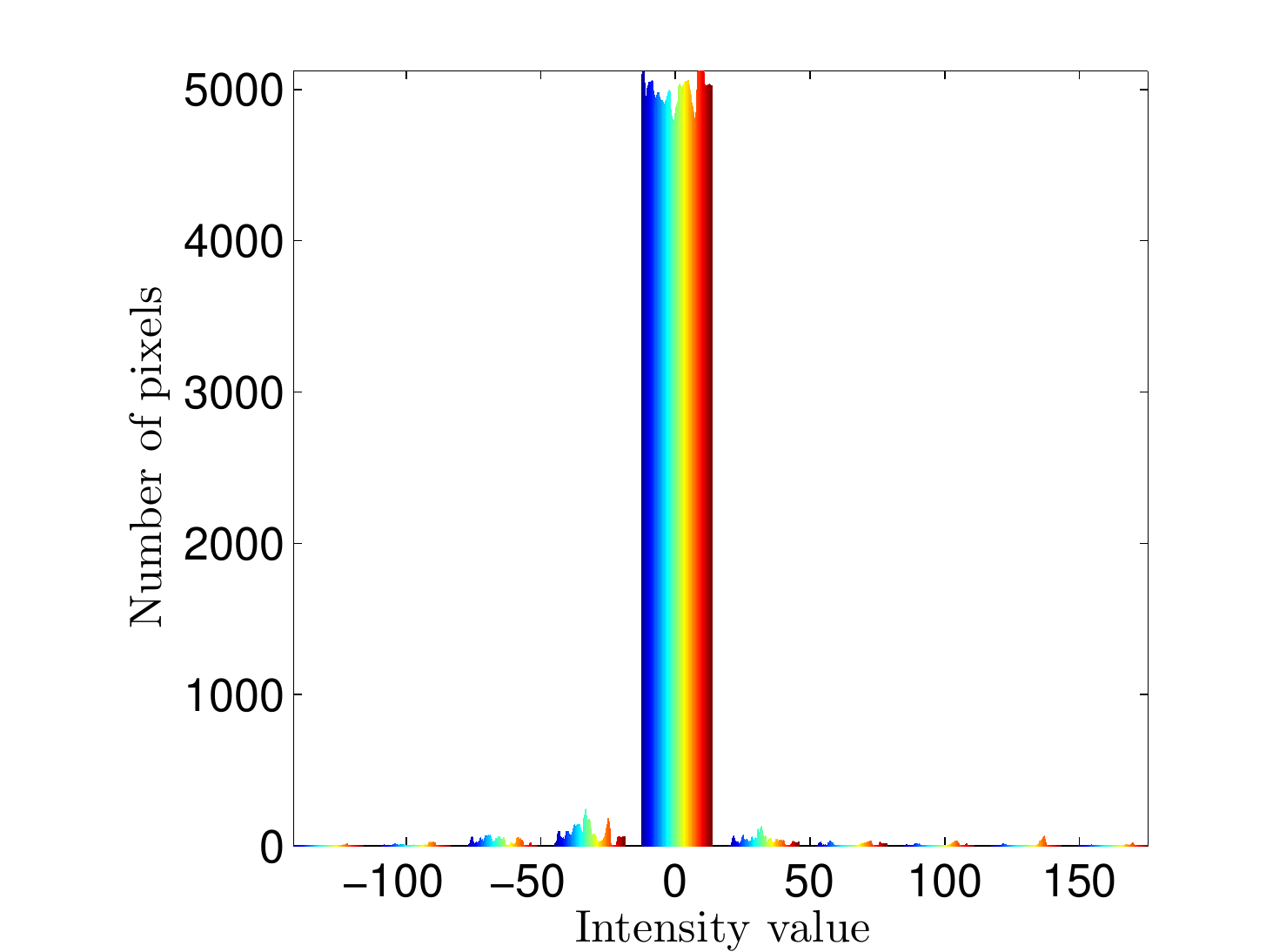}
		\caption{}
	\end{subfigure}
	\begin{subfigure}[b]{0.49\textwidth}
		\includegraphics[height=2.3in,width=1\textwidth]{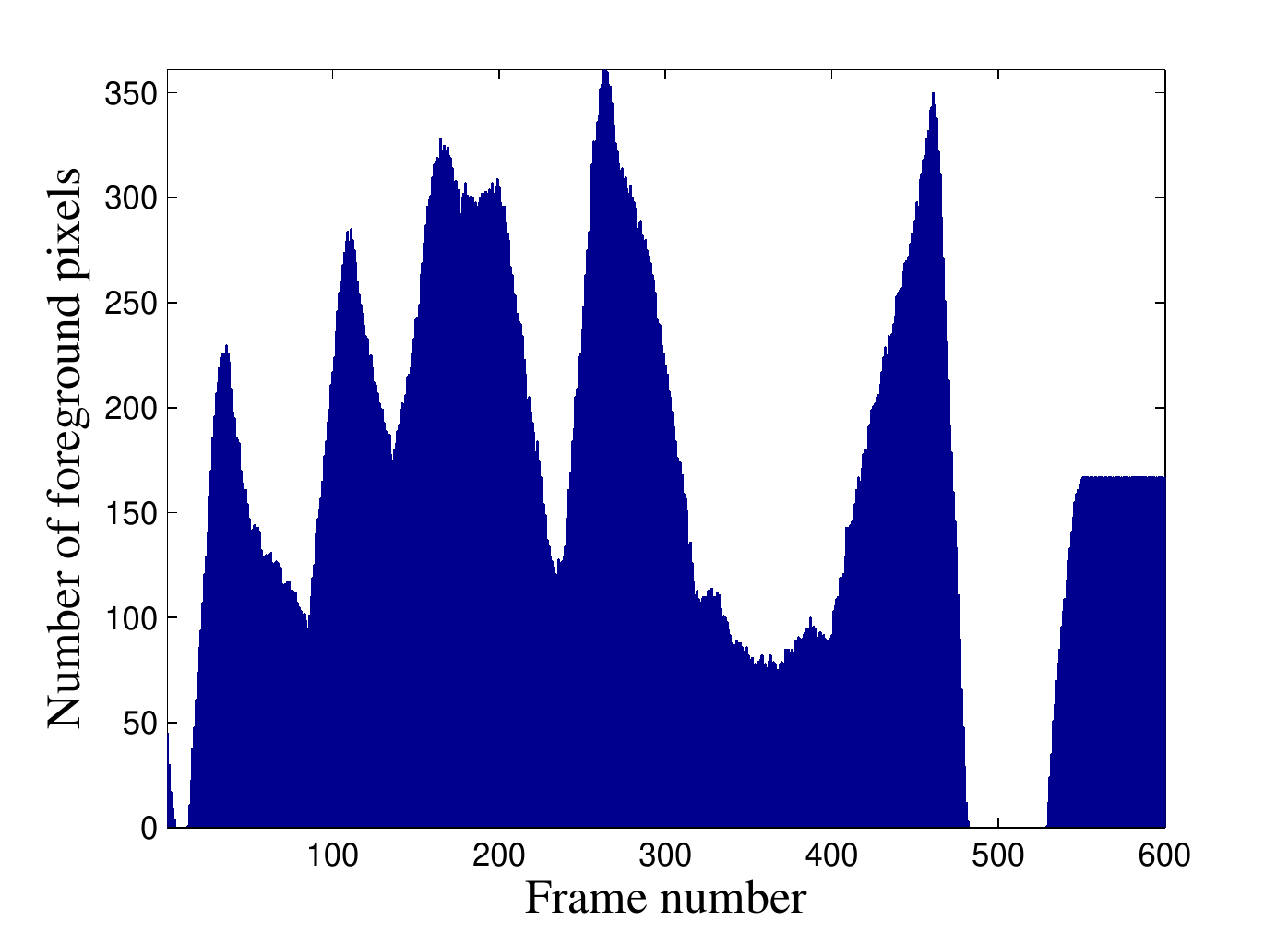}
		\caption{}
	\end{subfigure}
	\caption{Stuttgart video sequence {\it Basic} scenario:~(a)~Histogram to choose the threshold $\epsilon_1=31.2202$.~(b)~Sums of the logical values within each ground truth frame, where the frames with zero foreground pixels are purely background frames. See text for more details.}\label{gt_hist}
\end{figure}

We compare the performance of our algorithm to RPCA and SVT methods.~We set a uniform threshold $10^{-7}$ for each method. For iEALM and APG, the two prevalent algorithms for RPCA, we set $\lambda={1}/{\sqrt{{\rm max}\{m,n\}}}$, and for iEALM we choose $\mu=1.5, \rho=1.25$~\citep{APG,candeslimawright, LinChenMa}. To choose the right set of parameters for WSVT, we perform a grid search using a small holdout subset of frames. For WSVT, we set $\tau=4500$, $\mu = 5,\rho=1.1$ for a fixed weight matrix $W$. For SVT, we set $\tilde{\tau}=\tau/\mu$ since our method is equivalent to SVT for $W=I_n$.~Next, we show the effectiveness of the our WSVT and a mechanism for automatically estimating the weights from the data.

\subsubsection{Estimating the weight matrix $W$ in our WSVT}
We present a mechanism for estimating the weights from the data for WSVT. We use the heuristic that the data matrix $X$ can be comprised of two blocks $X_1$ and $X_2$ such that $X_1$ mainly contains the information about the background frames which have the least foreground movements. 
\begin{figure}
	\begin{minipage}{0.45\textwidth}
		\centering
		\begin{subfigure}{\textwidth}
			\includegraphics[width=\textwidth]{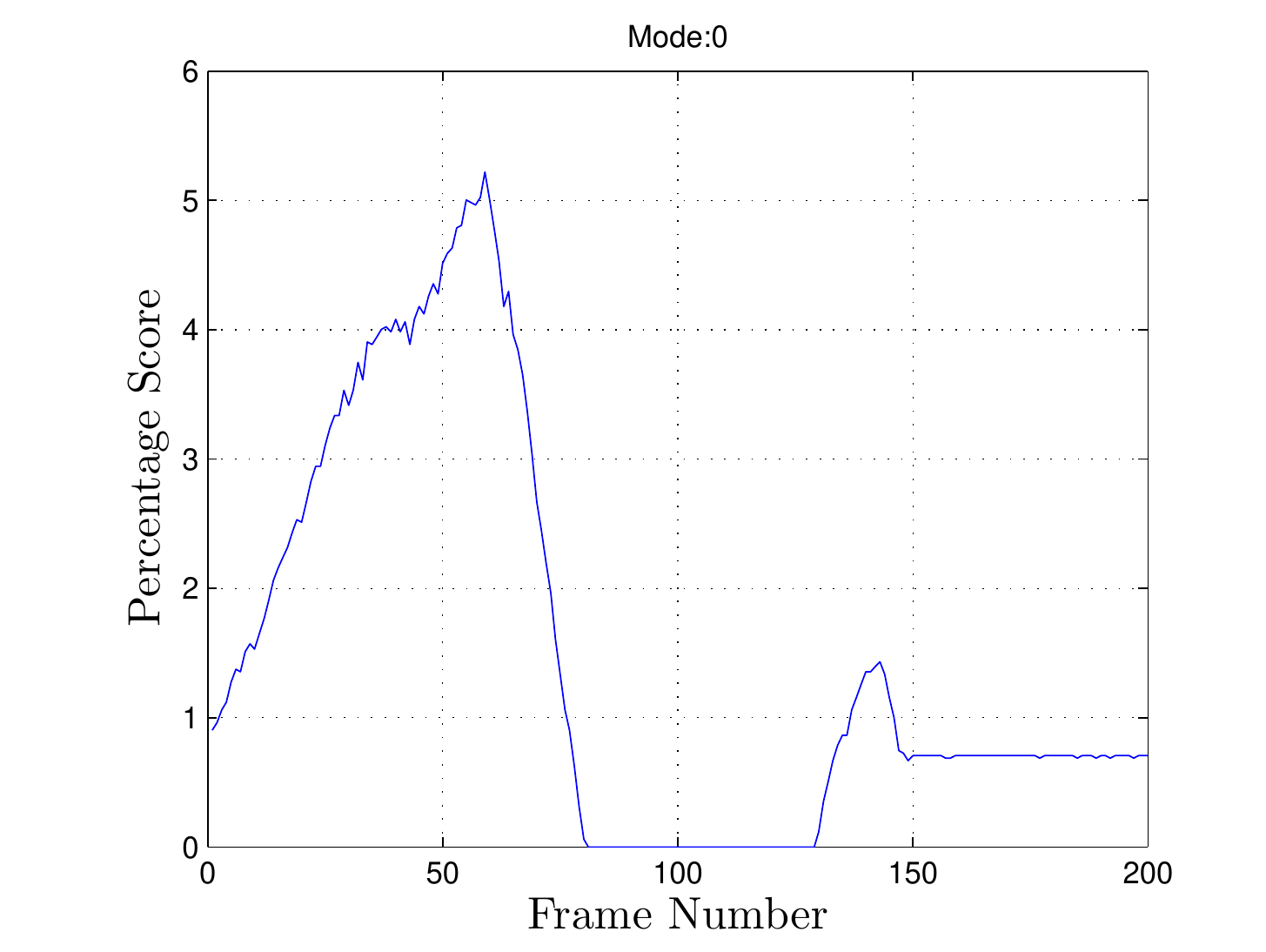}
			\caption{}\label{fig:leftA}
		\end{subfigure}
		\begin{subfigure}{\textwidth}
			\includegraphics[width=\textwidth]{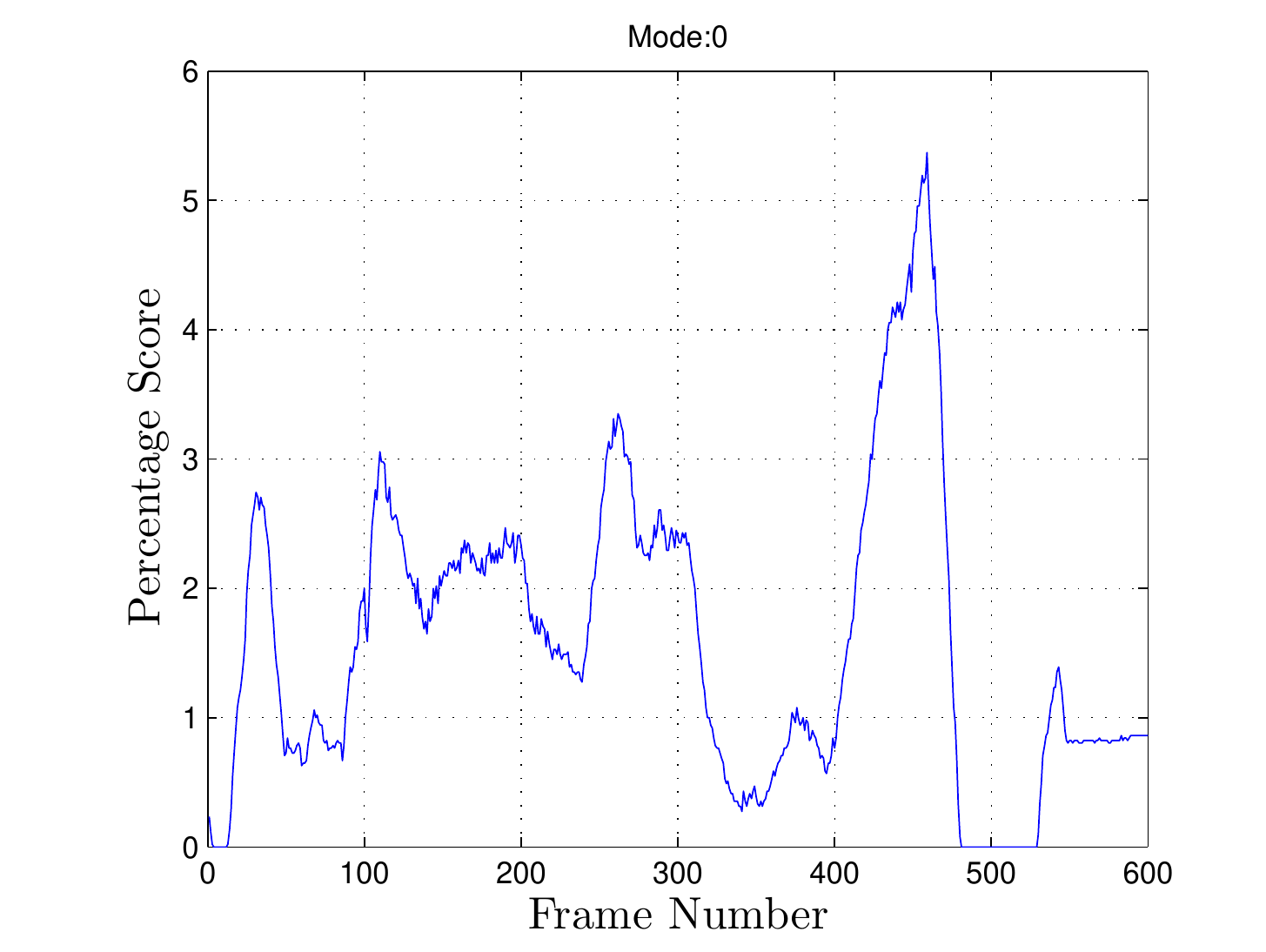}
			\caption{}\label{fig:leftB}
		\end{subfigure}
	\end{minipage}\hfill
	\begin{minipage}{0.45\textwidth}
		\centering
		\begin{subfigure}{\textwidth}
			\includegraphics[width=\textwidth]{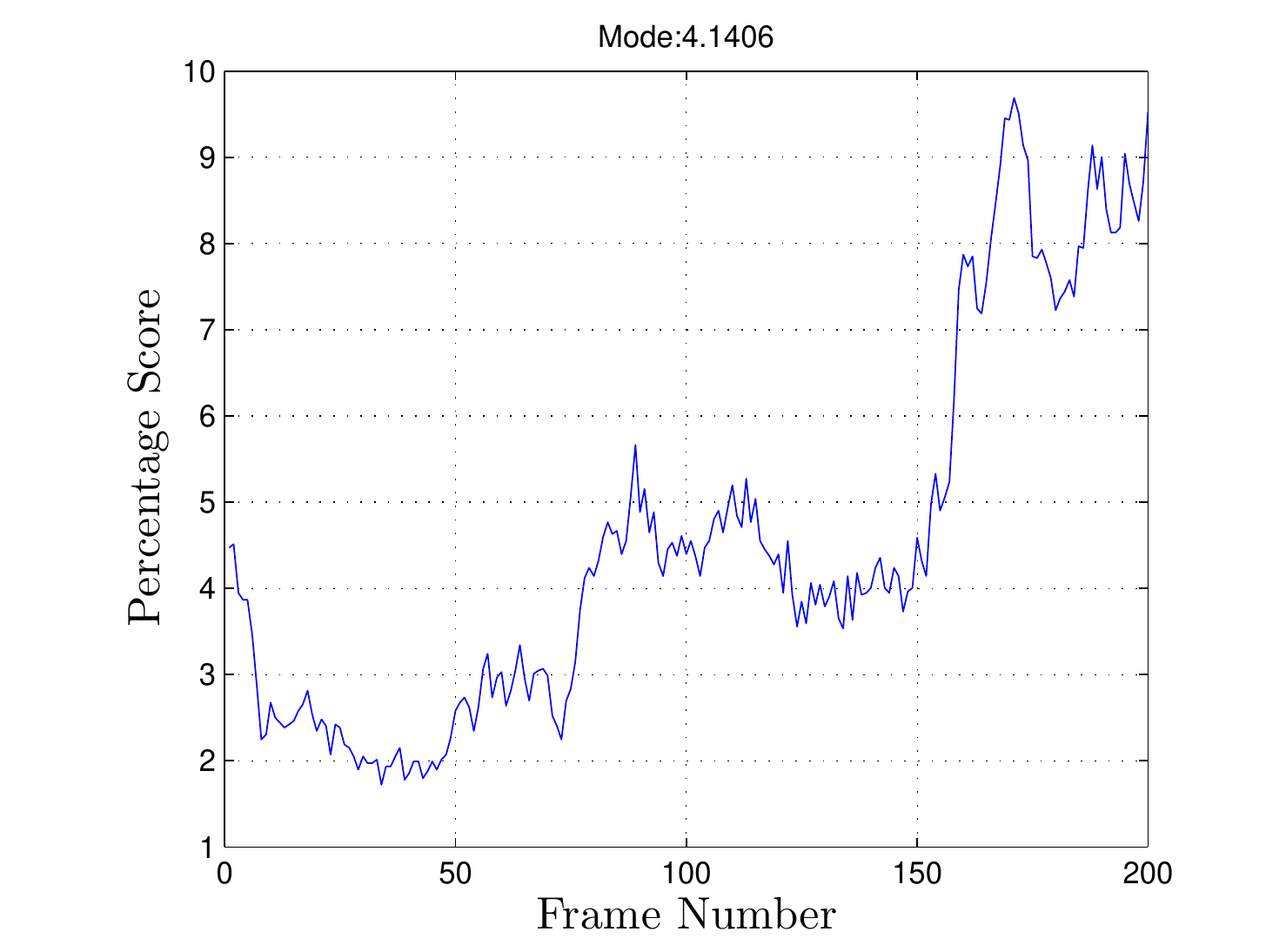}
			\caption{}\label{fig:rightA}
		\end{subfigure}
		\begin{subfigure}{\textwidth}
			\includegraphics[width=\textwidth]{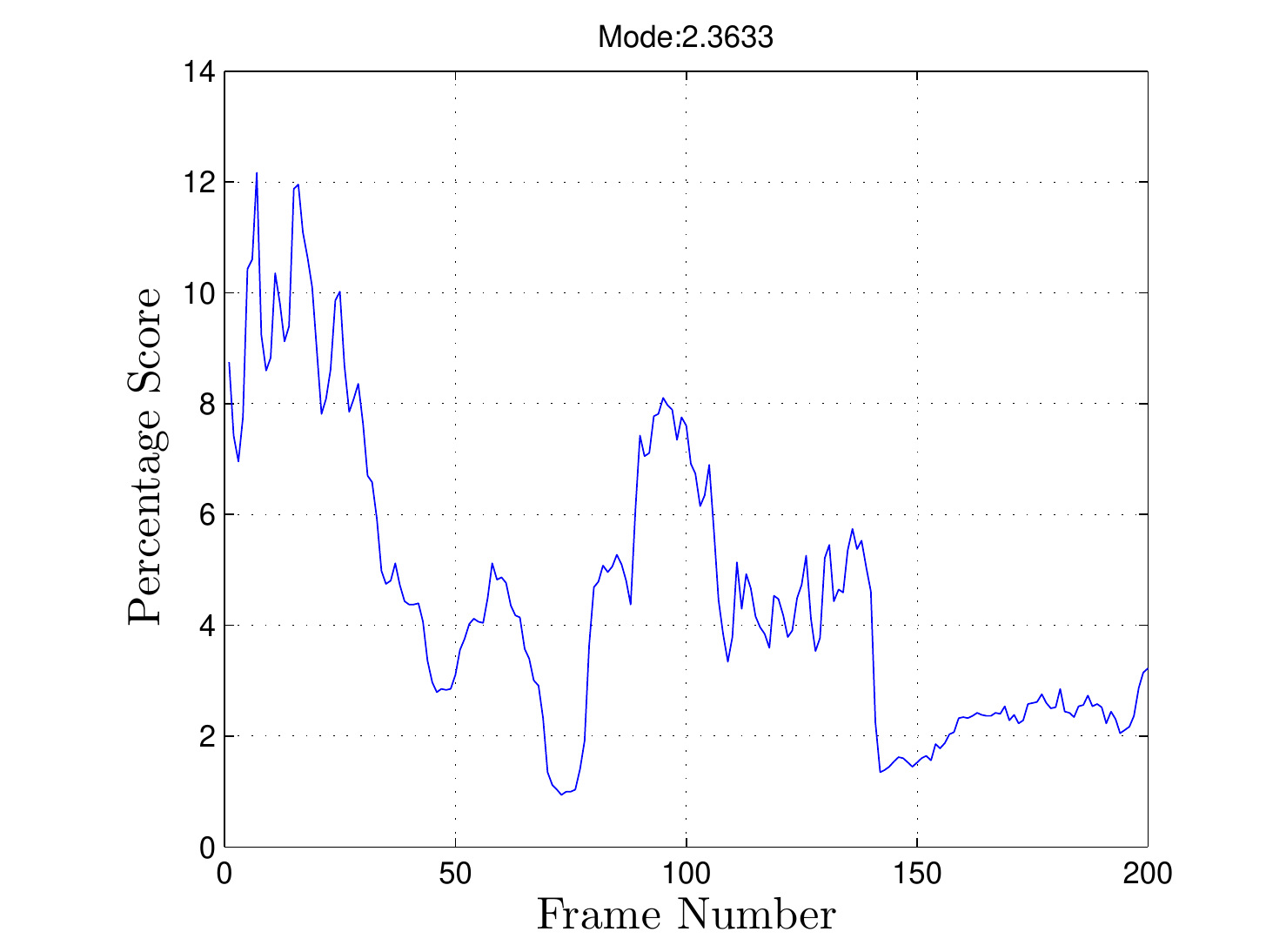}
			\caption{}
		\end{subfigure}
	\end{minipage}
	\caption{{\bf Robust weight learning.}~Frame number versus Percentage score for Stuttgart video sequence {\it Basic} scenario:~(a)~on last 200 frames, and~(b)~on the entire sequence.~Originally, there are 46 frames and 53 frames in the last 200 frames and the entire sequence respectively, with no foreground. We also show the percentage scores versus frame numbers on the first 200 frames for:~(c)~the fountain sequence, and (b)~the airport hall sequence. 
	}\label{perc_score}
\end{figure}

Therefore, we want to choose a large weight $\tilde{\lambda}$ corresponding to the frames of $X_1$. However, the changing illumination, reflection, and noise are typically also a part of those frames and pose a lot of challenges. We thus instead recover a low-rank matrix $B=(B_1\;\;B_2)$ with compatible block partition such that $B_1\approx X_1$. For this purpose, the main idea is to have a coarse estimation of the background using an identity weight matrix, infer the weights from the coarse estimation, and then use the inferred weights to refine the background.
We denote the test matrix as $T$,~and ground truth matrix as $G$. We borrow some notations from MATLAB to explain the experimental setup. Note that, the last 200 frames of the Stuttgart video sequence are the most challenging among the sequence, containing static foreground~(last 50 frames) along with moving foreground object and varying illumination.

We use our method with $W=I_n$ for 2 iterations on the frames, and then detect the initial foreground~$F_{In}$ and background $B_{In}$. It might not be the best practice to specify the number of foreground/background pixels manually for each test video sequence. Instead, we propose an automatic mechanism.~We plot the histogram of our initially detected foreground to determine the threshold $\epsilon_1$ of the intensity value. In our experiments on the Stuttgart video sequence {\it Basic} scenario, we pick $\epsilon_1=31.2202$, the second smallest value of $|(F_{In})_{ij}|$, where $|\cdot|$ denotes the absolute value~(see Figure~\ref{gt_hist}).

We replace everything below $\epsilon_1$ by 0 in $F_{In}$, and convert it into a logical matrix $LF_{In}$.~Arguably, for each such logical video frame, the number of pixels whose values are on $(+1)$ is a good indicator about whether the frame is mainly about the background.~Note that $\sum_{i=1}^m(LF_{IN})_{ij}$ is the $j$th column sum of $LF_{IN}$.~We convert $B_{In}$ directly to a logical matrix $LB_{In}$ and define percentage score as the ratio of total foreground and background pixels converted into percentage.~We calculate the percentage score of each video frame and choose the threshold $\epsilon_2$ as
$$ \epsilon_2:= {\rm mode}(\{\frac{\sum_i(LF_{IN})_{i1}}{\sum_i(LB_{IN})_{i1}}\times100,\frac{\sum_i(LF_{IN})_{i2}}{\sum_i(LB_{IN})_{i2}}\times100,\cdots, \frac{\sum_i(LF_{IN})_{in}}{\sum_i(LB_{IN})_{in}}\times100\}),
$$
where $\{\frac{\sum_i(LF_{IN})_{ij}}{\sum_i(LB_{IN})_{ij}}\times100\}_{j=1}^{n}$ are the percentage score of each frame. Since the foreground pixels are relatively smaller size compared to the background, heuristically the possible contender of the pure background frame indexes will have percentage score less than $\epsilon_2.$~Therefore, the frame indexes with least foreground movement are chosen from the following set:
$$
I = \{i:(\frac{\sum_i(LF_{IN})_{i1}}{\sum_i(LB_{IN})_{i1}}\times100,\frac{\sum_i(LF_{IN})_{i2}}{\sum_i(LB_{IN})_{i2}}\times100,\cdots, \frac{\sum_i(LF_{IN})_{in}}{\sum_i(LB_{IN})_{in}}\times100)\le \epsilon_2\}.
$$
Figure~\ref{gt_hist}~shows the initially estimated foreground histogram and the sums of the ground truth frames of the {\it Basic} scenario of the Stuttgart video sequence. 
Figure~\ref{perc_score} demonstrates the percentage score plot for the {\it Basic} scenario of the Stuttgart video sequence, the fountain sequence, and the airport hall sequence.~Originally, for the {\it Basic} scenario of the Stuttgart video sequence, there are 48 and 57 frames respectively in the last 200 frames and the entire sequence that have less than 5 foreground pixels.~Using the percentage score, our method picks up 49 and 58 frame indexes respectively.~Moreover,~comparing~Figure~\ref{perc_score}(a) and (b) with the ground truth frames in Figures \ref{gt_hist}~(b),~we see the effectiveness of the process in picking up the right background frame indexes on the Stuttgart video sequence.

With the automatically inferred diagonal weight matrix $W$, whose entries corresponding to the coarsely estimated background information are assigned a large value $\tilde{\lambda}$ and other diagonal entries are 1's, we conduct the remaining experiments. 

\subsubsection{Convergence of the algorithm}
In this subsection, we will show the convergence of our WSVT algorithm.~For a given $\epsilon>0$, the main stopping criteria of our WSVT algorithm is $|L_{k+1}-L_k|<\epsilon$ or if it reaches the maximum iteration.~To demonstrate the convergence of our algorithm as claimed in Theorem~\ref{theorem_3_3}, we run it on the {\it Basic} scenario of the Stuttgart artificial video sequence.~The weights were chosen using the idea explained in Subsection 4.1.1. We choose $\displaystyle{\tilde{\lambda}\in\{1,5,10,20\}}$ and $\epsilon$ is set to $10^{-7}$. Recall that in Theorem~\ref{theorem_3_3} we present $\|D_k-C_kW^{-1}\|\le O(\frac{1}{\mu_k})$ and $|L_{k+1}-L_k|\le O(\frac{1}{\mu_k})$ as $\mu_k\to\infty$.~We see that our proposed algorithm converges and the bounds we present on the iterates $\{C_k,D_k\}$ and the reconstruction error $L_k$ in Theorem~\ref{theorem_3_3} are valid.~To conclude, in Figure~\ref{conv}, we show that for any $\tilde{\lambda}>0$, there exists $\alpha,\beta\in\mathbb{R}$ such that $\|D_k-C_kW^{-1}\|_F\le\alpha/\mu_k$ and $|L_{k+1}-L_k|\le\beta/\mu_k$ as $\mu_k\to\infty,$ for $k=1,2,\cdots$.
 \begin{figure}
 	\centering
 	\begin{subfigure}{0.5\textwidth}
 		\includegraphics[height=2.4in,width=3.27in]{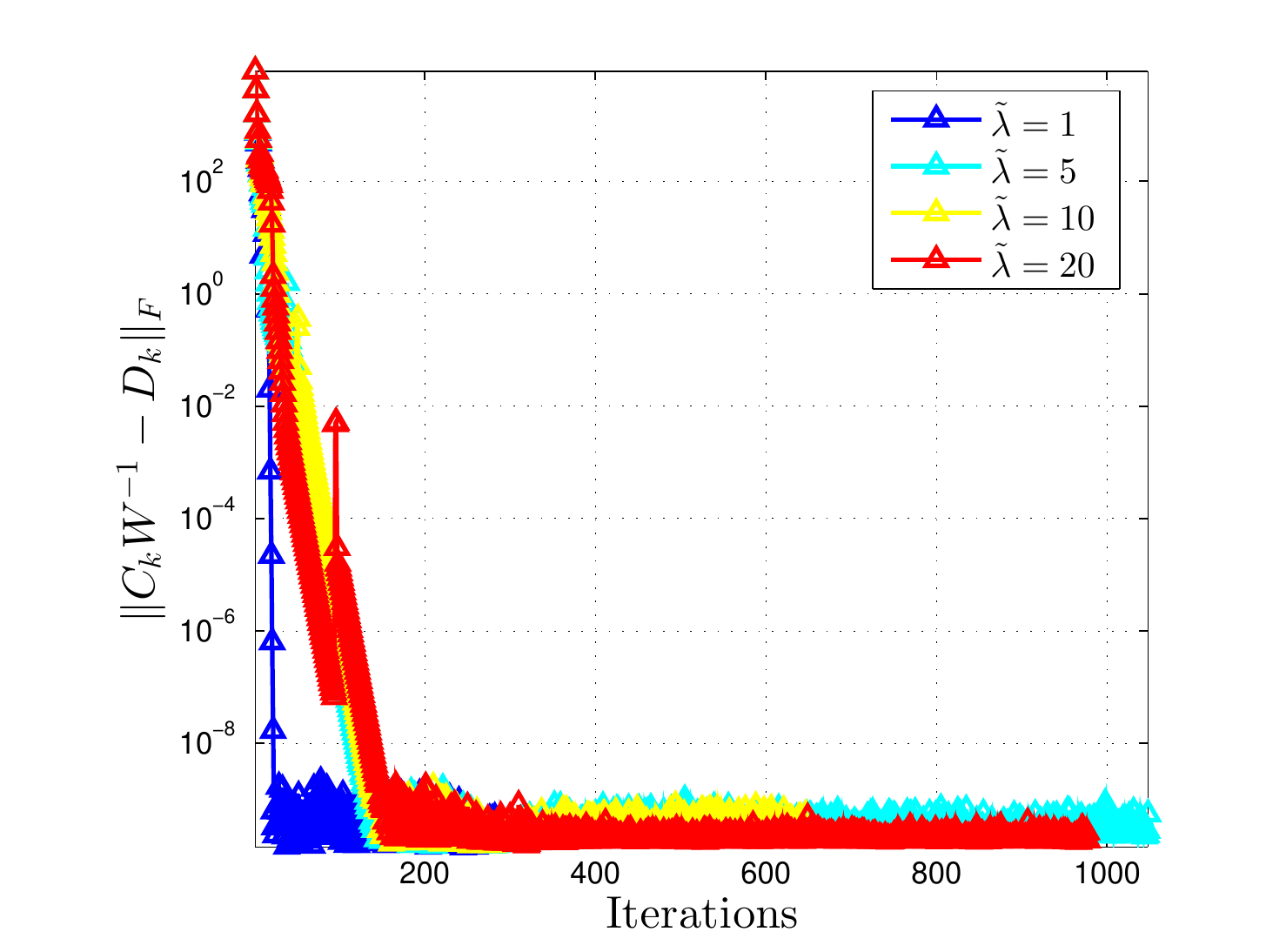}
 		\caption{}
 	\end{subfigure}\hspace{-1ex}
 	\begin{subfigure}{0.5\textwidth}
 		\includegraphics[height=2.4in]{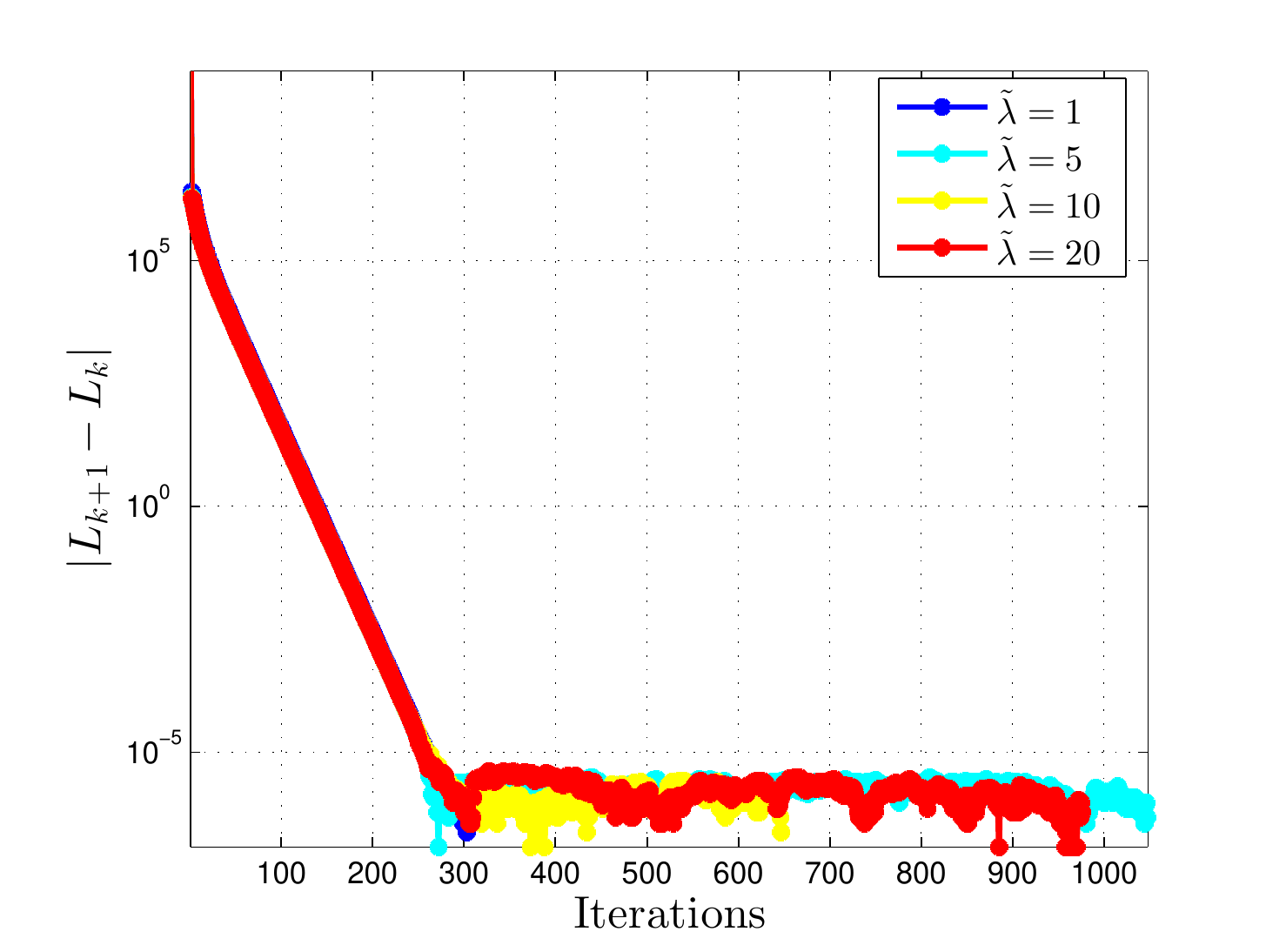}
 		\caption{}
 	\end{subfigure}
 	\caption{{\bf Convergence~of~the~algorithm.}~For $\displaystyle{\tilde{\lambda}\in\{1,5,10,20\}}$:~(a)~Iterations versus $\|D_k-C_kW^{-1}\|_F$, and~(b)~Iterations versus $|L_{k+1}-L_k|$.}\label{conv}
 \end{figure}

\subsubsection{Qualitative and Quantitative analysis}

In this section, we perform rigorous qualitative and quantitative comparison between WSVT, SVT, and RPCA algorithms on three different video sequences:~Stuttgart artificial video sequence, the airport hall sequence, and the fountain sequence.~For the quantitative comparison between different methods, we only use Stuttgart artificial video sequence. We use three different metrics for quantitative comparison: The receiver and operating characteristic~(ROC) curve,~peak signal-to-noise ratio~(PSNR) and the mean structural similarity index~(MSSIM)~\citep{mssim}.~They are all methods for measuring the similarity between two images -- in our scenario, the ground truth and estimated foreground.
\begin{figure}
	\centering 
	\includegraphics[width=\textwidth ]{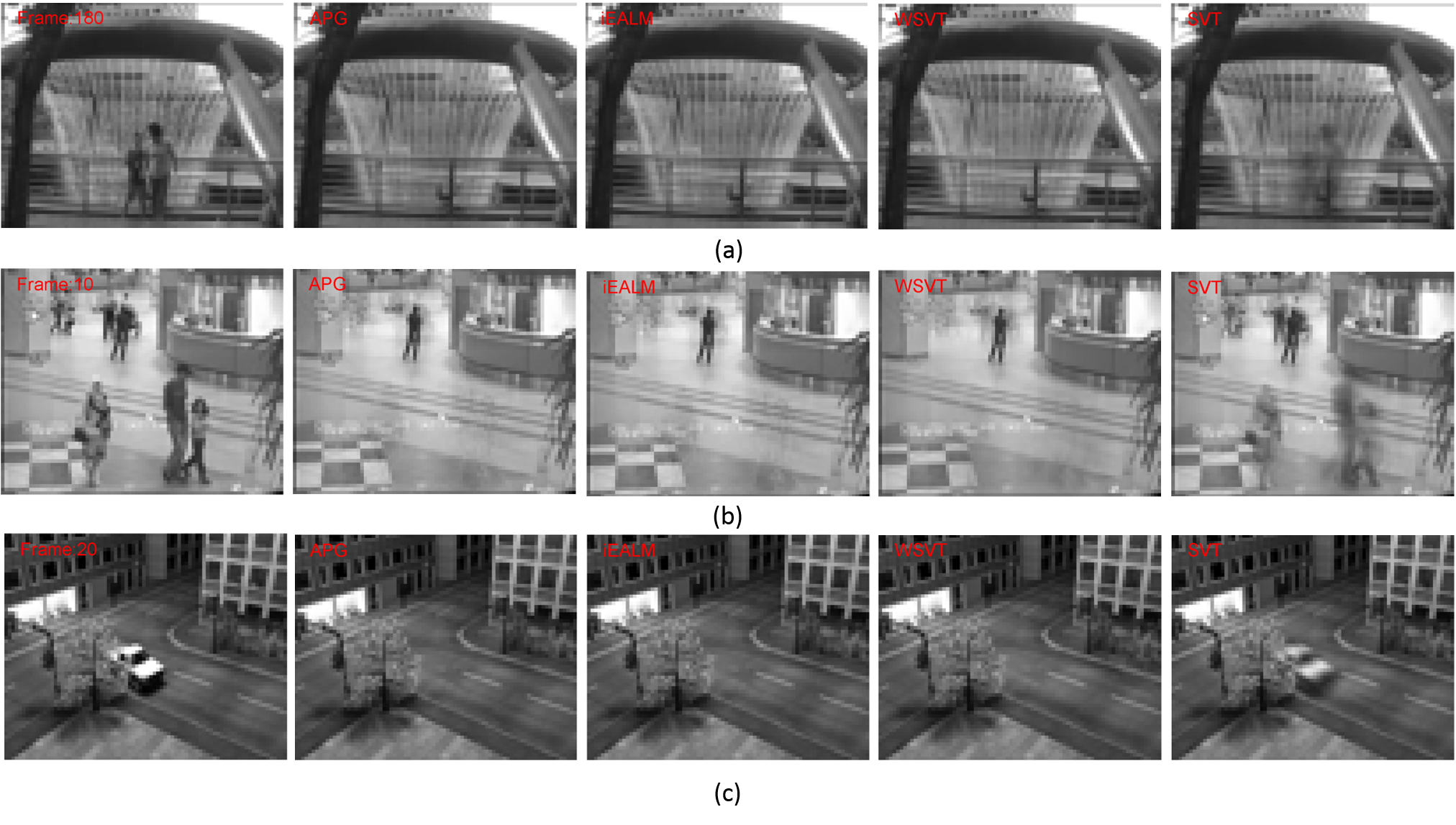}
	\caption{{\bf Qualitative analysis:}~From left to right:~Original, APG, iEALM, our WSVT, and SVT.~Background estimation results on~(from top to bottom):~(a)~fountain sequence, frame number 180 with static and dynamic foreground;~(b)~airport sequence, frame number 10 with static and dynamic foreground; (c)~Stuttgart video sequence~{\it Basic} scenario,~frame number 420 with dynamic foreground.}\label{qual_bg}
\end{figure}  

\begin{figure}
	\centering  \includegraphics[width=\textwidth ]{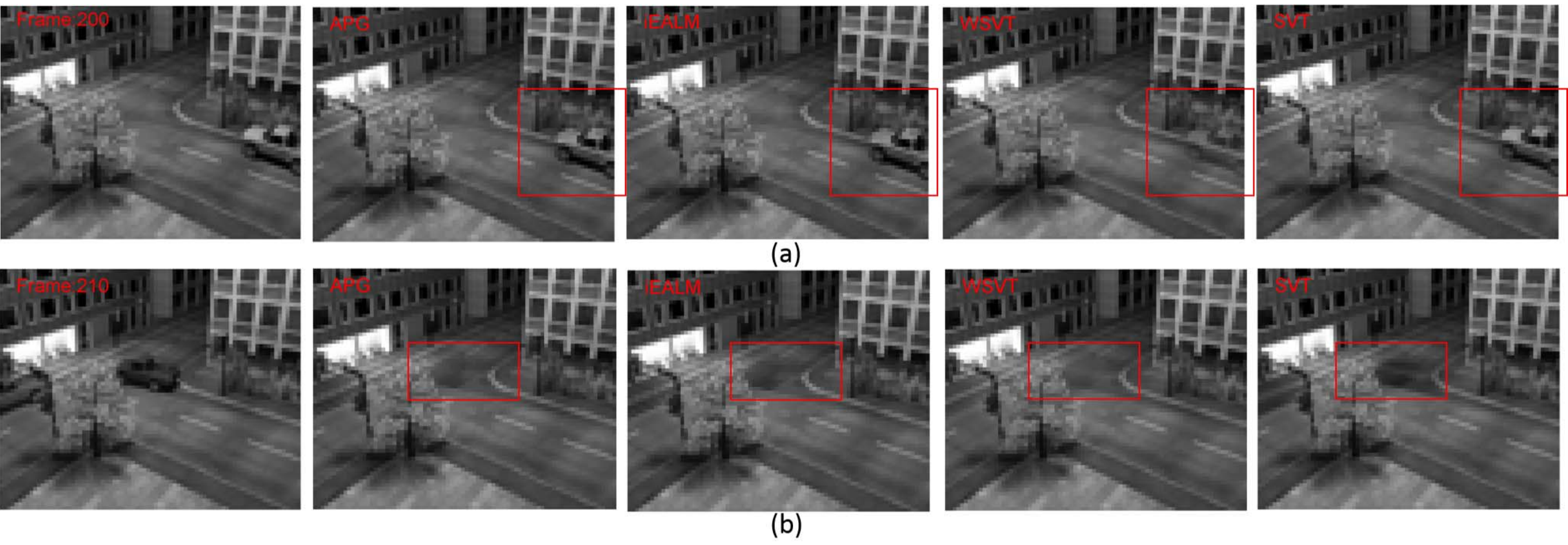}
	\caption{{\bf Qualitative analysis:}~From left to right:~Original, APG, iEALM, WSVT,~and SVT on Stuttgart video sequence {\it Basic} scenario:~(a) frame 600 with static foreground, methods were tested on last 200 frames;~(b)~frame 210 with dynamic foreground, methods were tested on 600 frames. WSVT provides a superior background estimation.}\label{qual_bg_2}
\end{figure}

Initially, we test each method on 200 resized video frames and the qualitative results are shown in Figure~\ref{qual_bg}.~We employ the process defined in Section 4.1.1 to adaptively choose the weighted frame indexes for WSVT.~Next, we test our method on the entire Stuttgart video sequence and compare its performance with the other unweighted low-rank methods.~Unless specified, a weight $\tilde{\lambda}=5$ is used to show the qualitative results for the WSVT algorithm in Figure~\ref{qual_bg} and~\ref{qual_bg_2}.~It is evident from Figure~\ref{qual_bg} that WSVT outperforms SVT and recovers the background as effectively as RPCA methods.~However, in Figure \ref{qual_bg_2} where there are both static and dynamic foreground objects, WSVT shows superior performance over all other methods including RPCA algorithms.

\paragraph{Quantitative analysis.}
\begin{figure}
	\centering
	\begin{subfigure}[b]{0.55\textwidth}
		\includegraphics[height=2.4in]{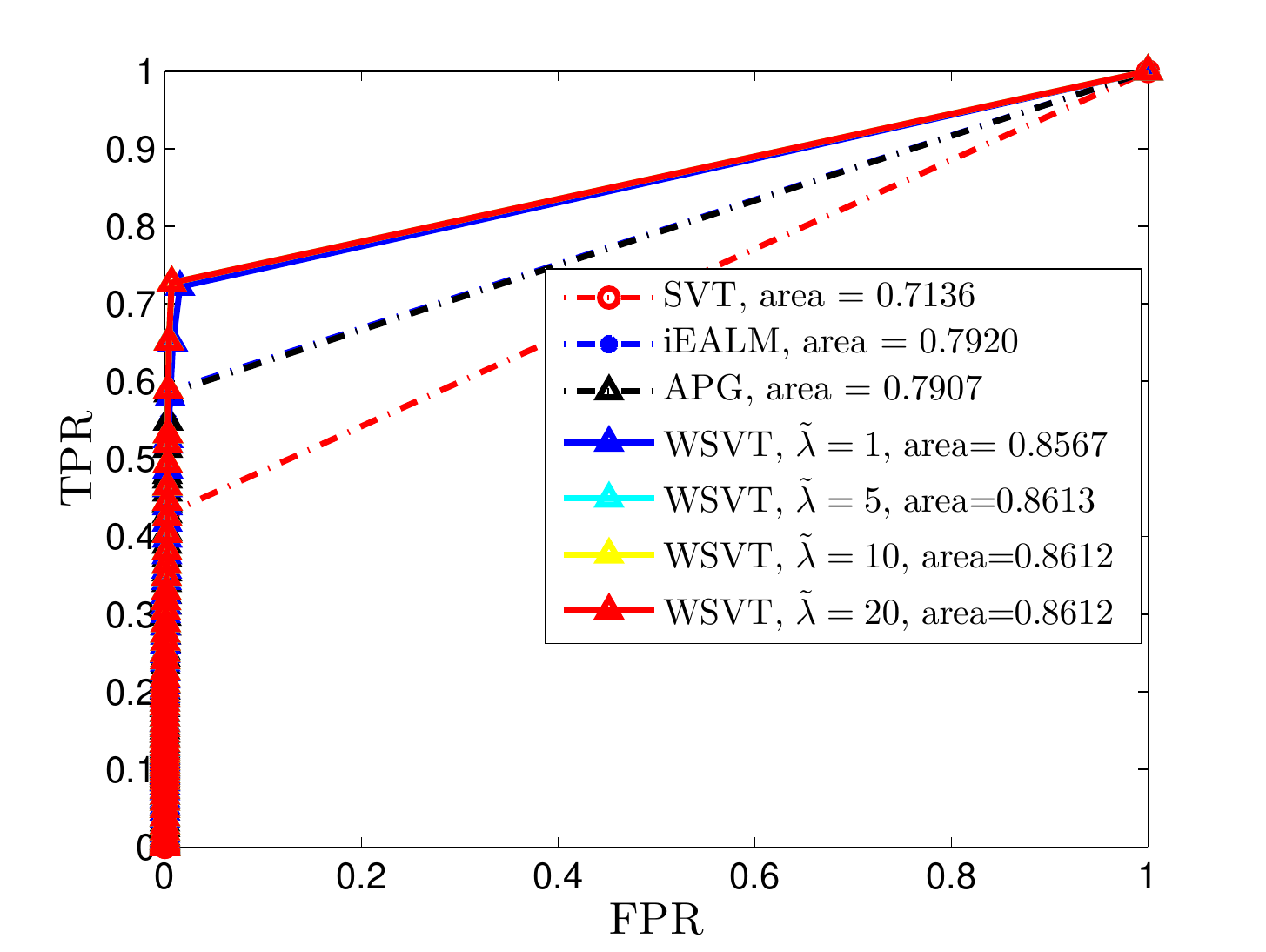}
		\caption{}
	\end{subfigure}%
	\begin{subfigure}[b]{0.55\textwidth}
		\includegraphics[height=2.4in]{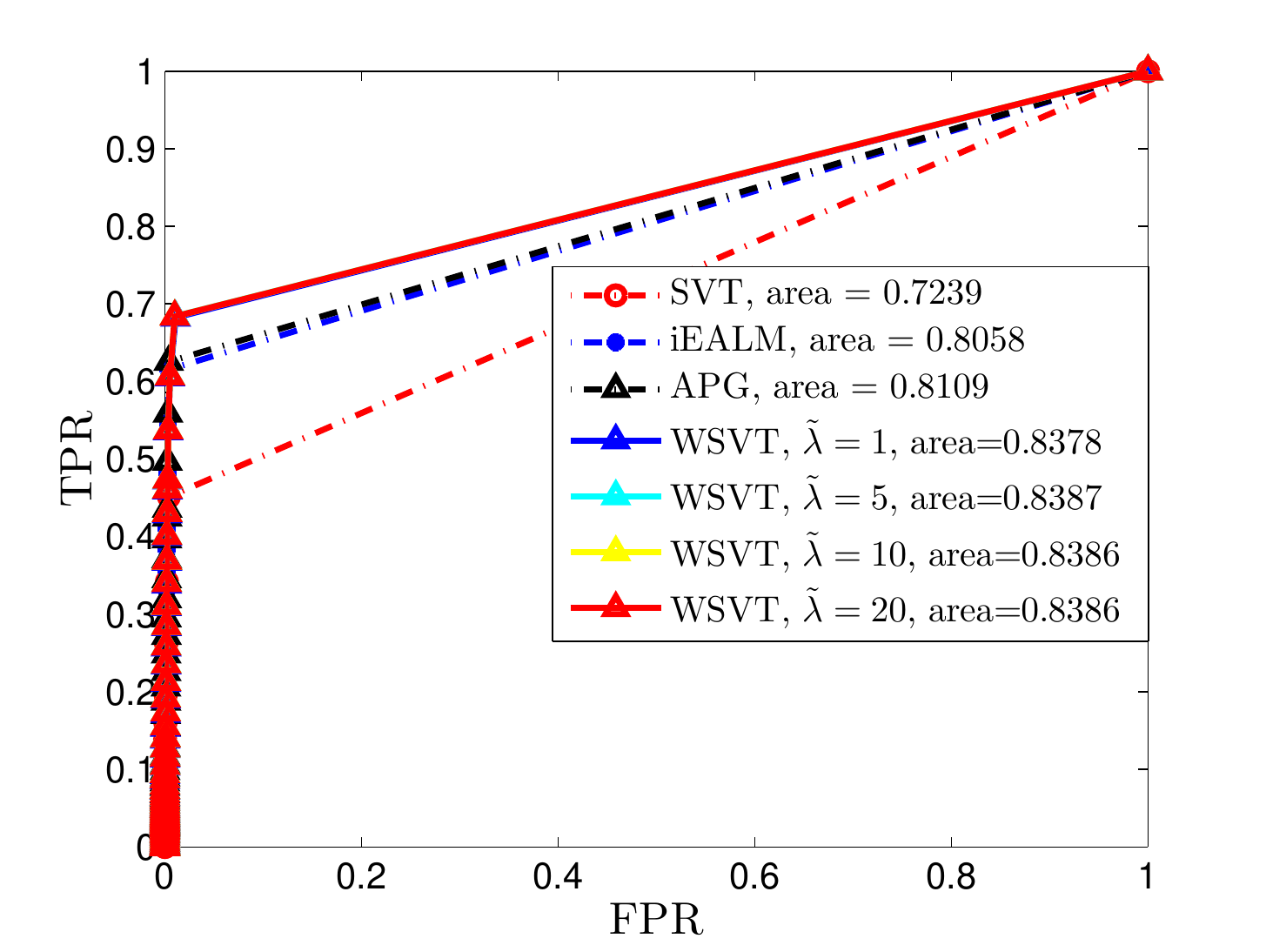}
		\caption{}
	\end{subfigure}
	\caption{ROC curve for the methods WSVT, SVT, iEALM, and APG on {\it Basic} scenario of the Stuttgart sequence:(a)~200 frames, and~(b)~600 frames. For WSVT we choose $\tilde{\lambda}\in\{1,5,10,20\}$. The performance gain by WSVT compare to iEALM, APG, and SVT are: 8.92\%,~8.74\%, and 20.68\% respectively on 200 frames~(with static foreground), and 4.07\%,~3.42\%, and 15.85\% respectively on 600 frames.}\label{roc_new}
\end{figure}
Note that WSVT uniformly removes the noise~(for example, the changing illumination, reflection on the buildings, and movement of the leaves of the tree for the Stuttgart sequence) from each video sequence. 
Inspired by the above observation, we propose a nonuniform threshold vector to plot the ROC curves and compare between the methods using the same metric. In Figure~\ref{roc_new}, we provide quantitative comparisons between the methods using this non-uniform threshold vector \verb+[0,15,20,25,30,31:2.5:255]+.~This way we can reduce the number of false negatives and increase the number of true positives detected by WSVT as it appears in Figure~\ref{qual_bg},~\ref{qual_bg_2}~. To conclude, WSVT has better quantitative and qualitative results when there is a static foreground in the video sequence.
\begin{figure}
	\centering
	\begin{subfigure}[b]{0.55\textwidth}
		\includegraphics[height=2.4in]{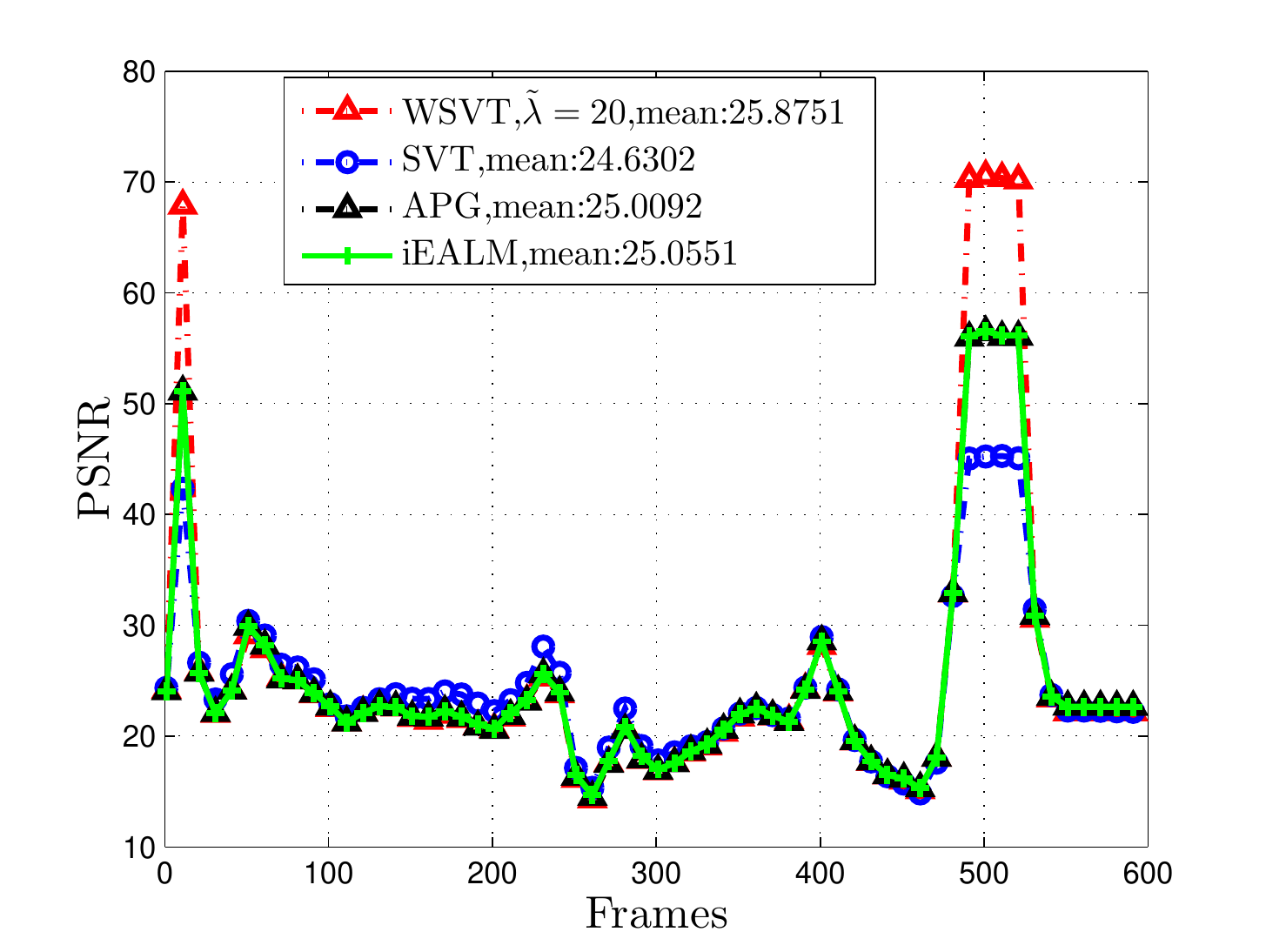}
		\caption{}
	\end{subfigure}%
	\begin{subfigure}[b]{0.55\textwidth}
		\includegraphics[height=2.4in]{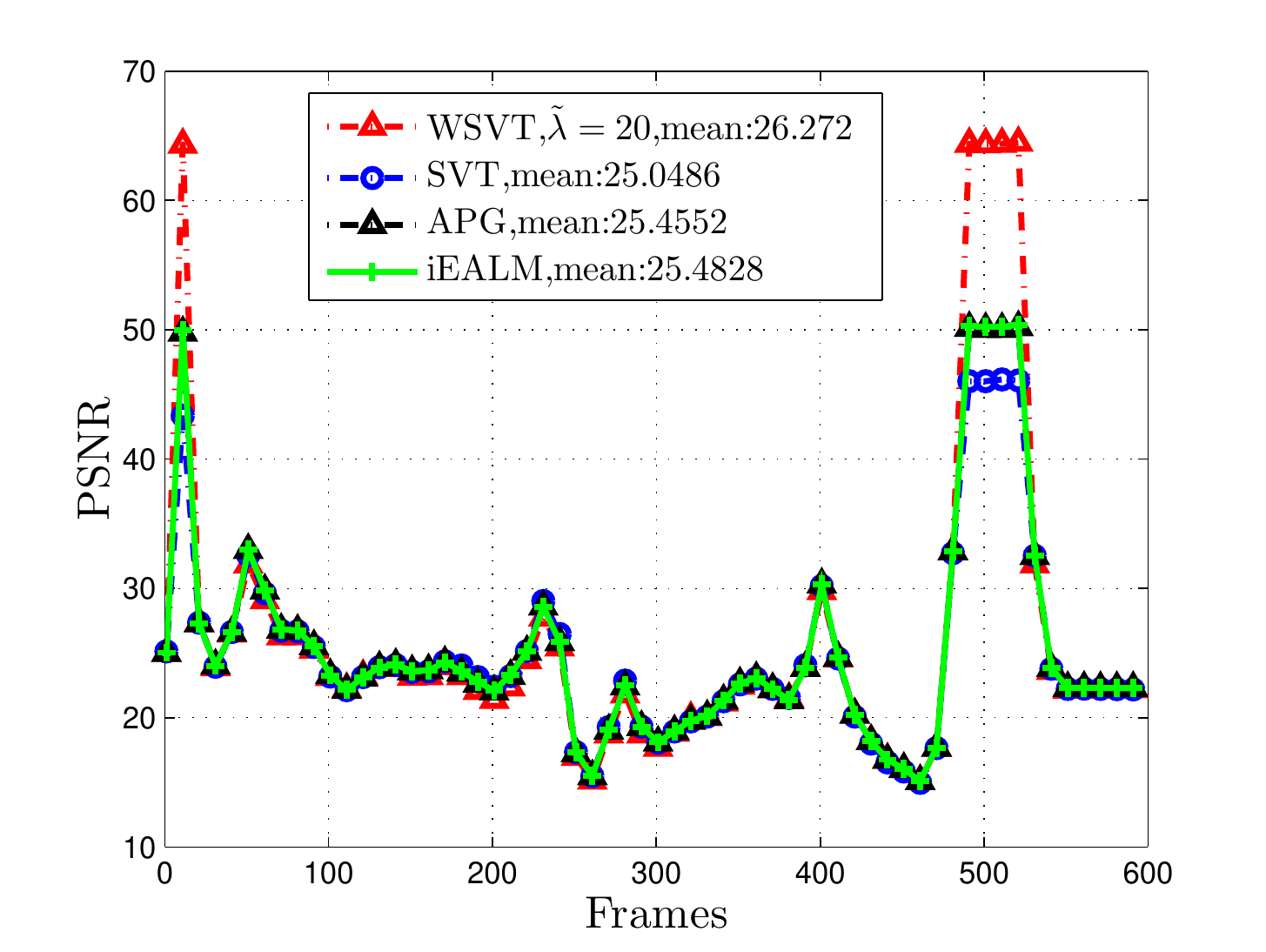}
		\caption{}
	\end{subfigure}
	\caption{PSNR of each video frame for WSVT, SVT, iEALM, and APG~on:(a)~{\it Basic} scenario, and~(b)~{\it Noisy night} scenario.~For WSVT we choose $\tilde{\lambda}=20$.~In both scenarios WSVT has increased PSNR when a weight is introduced corresponding to the frames with least foreground movement.}\label{psnr}
\end{figure}

A robust background estimation model used for surveillance may efficiently deal with the dynamic foreground objects present in the video sequence.~Additionally, it is expected to handle several other challenges, which include, but are not limited to: gradual or sudden change of illumination, a dynamic background containing non-stationary objects and a static foreground, camouflage, and sensor noise or compression artifacts.~The initial success of the WSVT in Figure~\ref{roc_new} on the {\it Basic} scenario of the Stuttgart sequence motivated us to demonstrate a more rigorous quantitative measure of the recovered foreground obtained by different background estimation models on a more challenging scenario. For this purpose, we test the effectiveness of the models on the {\it Noisy night} scenario of the Stuttgart video sequence.~This is a low-contrast nighttime scene, with increased sensor noise, camouflage, and sudden illumination change and has 600 frames with identical foreground and background objects as in the {\it Basic} scenario. Now, we demonstrate the quantitative analysis results of different methods on these two scenarios using PSNR and SSIM.~We use the same set of parameters for every model and follow Section 4.1.1 to choose weighted frames for WSVT. 
\begin{figure}
	\centering
	\begin{subfigure}[b]{0.55\textwidth}
		\includegraphics[height=2.4in]{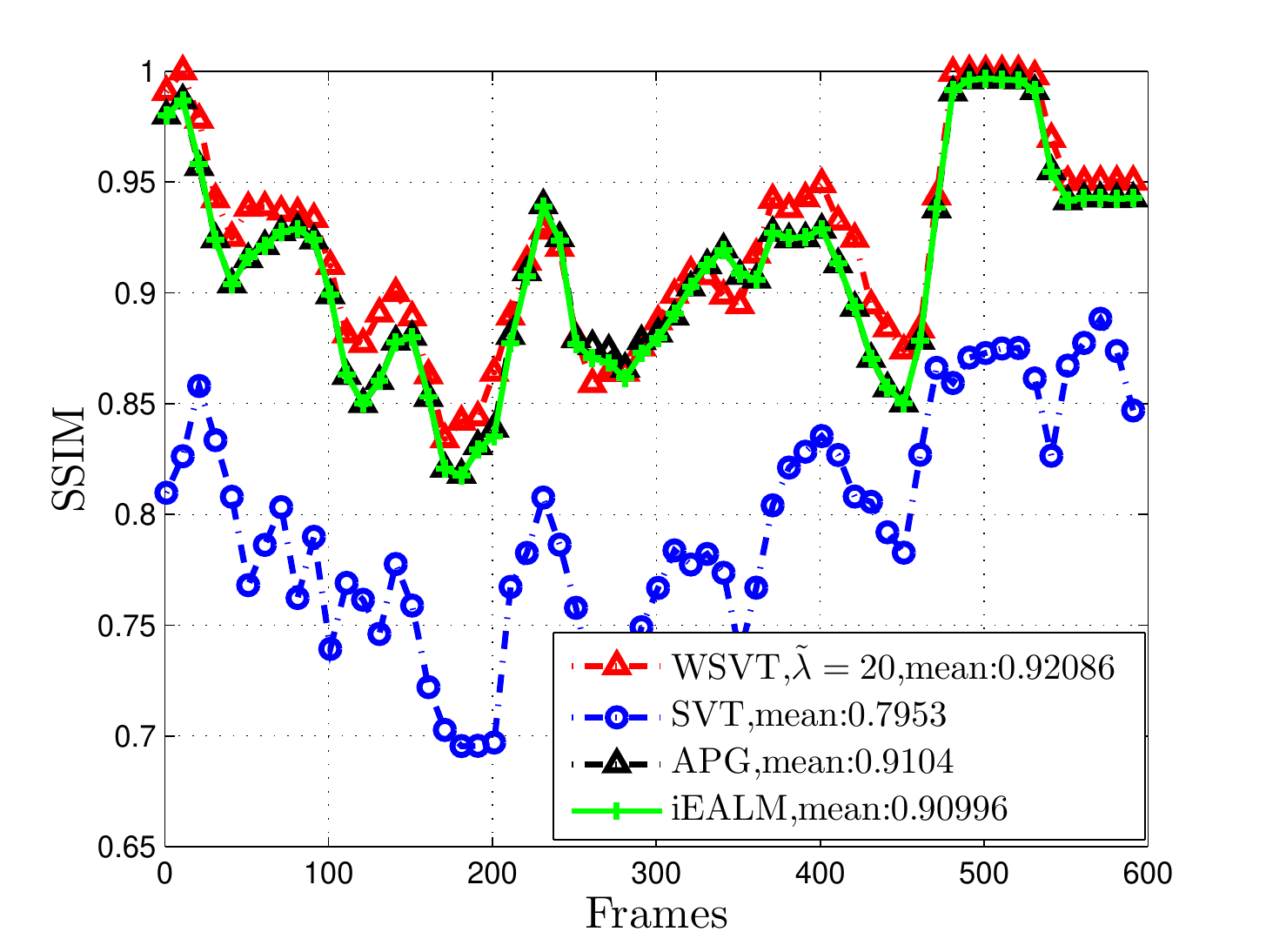}
		\caption{}
	\end{subfigure}%
	\begin{subfigure}[b]{0.55\textwidth}
		\includegraphics[height=2.4in]{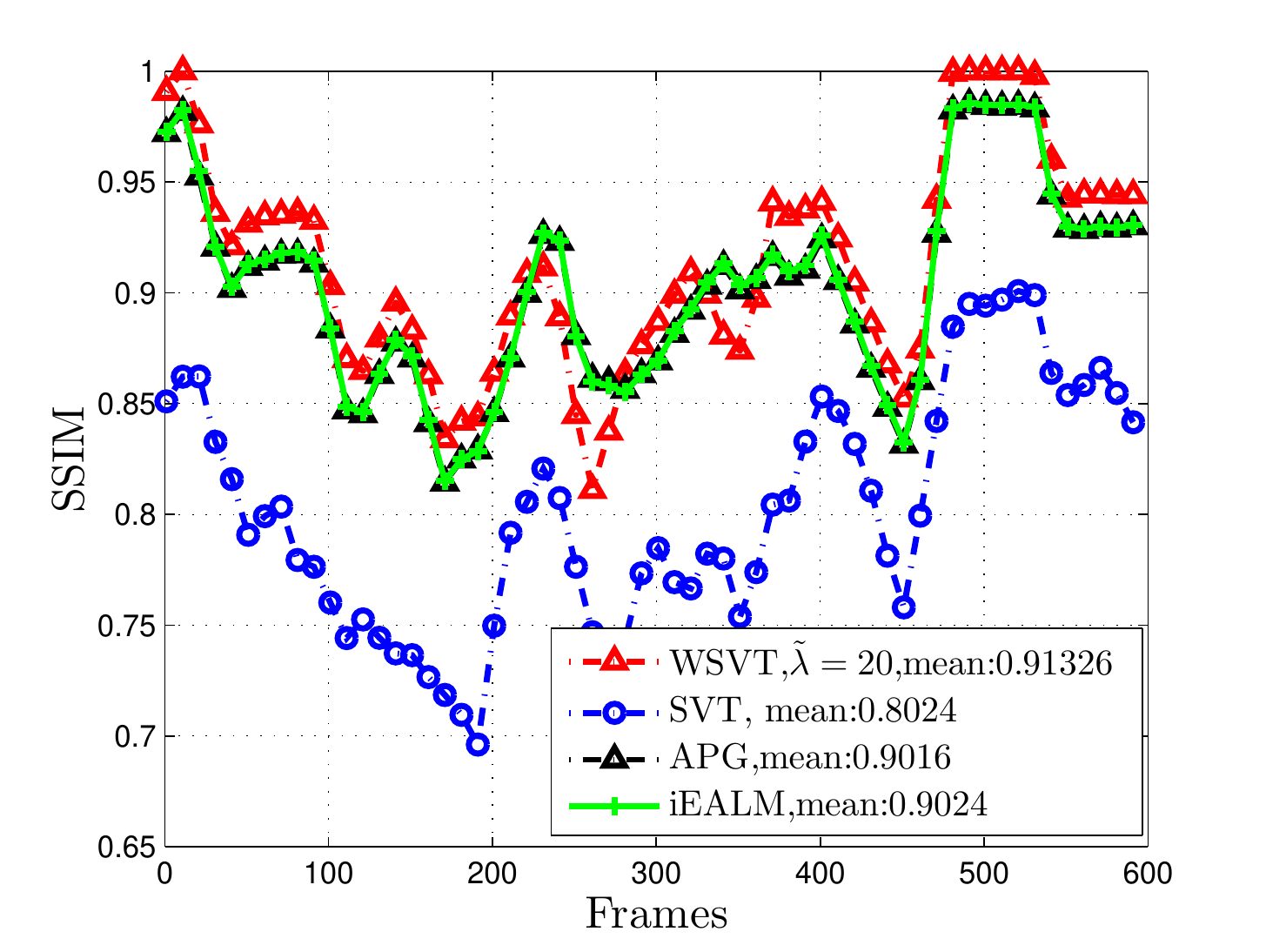}
		\caption{}
	\end{subfigure}
	\caption{Mean SSIM for different methods on 600 frames of Stuttgart sequence:~(a)~{\it Basic} scenario,~(b)~{\it Noisy night} scenario.~For WSVT we choose $\tilde{\lambda}=20$.~Indeed WSVT with weight outperforms other methods.}\label{mssim}
\end{figure}

PSNR is defined as $10log_{10}$ of the ratio of the peak signal energy to the mean square error~(MSE) observed between the processed video signal and the original video signal.~If $E(:,i)$ denotes each reconstructed vectorized foreground frame in the video sequence and $G(:,i)$ be the corresponding ground truth frame, then PSNR is defined as $10log_{10}\frac{{\rm M}_I^2}{{\rm MSE}}$, where ${\rm MSE}= \frac{1}{mn}\|E(:,i)-G(:,i)\|_2^2$ and ${\rm M}_I$ is the maximum possible pixel value of the image. In our case the pixels are represented using 8 bits per sample, and therefore, ${\rm M}_I$ is 255. The proposal is that the higher the PSNR, the better degraded image has been reconstructed to match the original image and the better the reconstructive algorithm.  This would occur because we wish to minimize the MSE between images with respect the maximum signal value of the image. For a reconstructed image with 8 bits bit depth, the PSNR are between 30 and 50 dB, where the higher is the better.

\begin{figure}[H]
	\includegraphics[width=\textwidth]{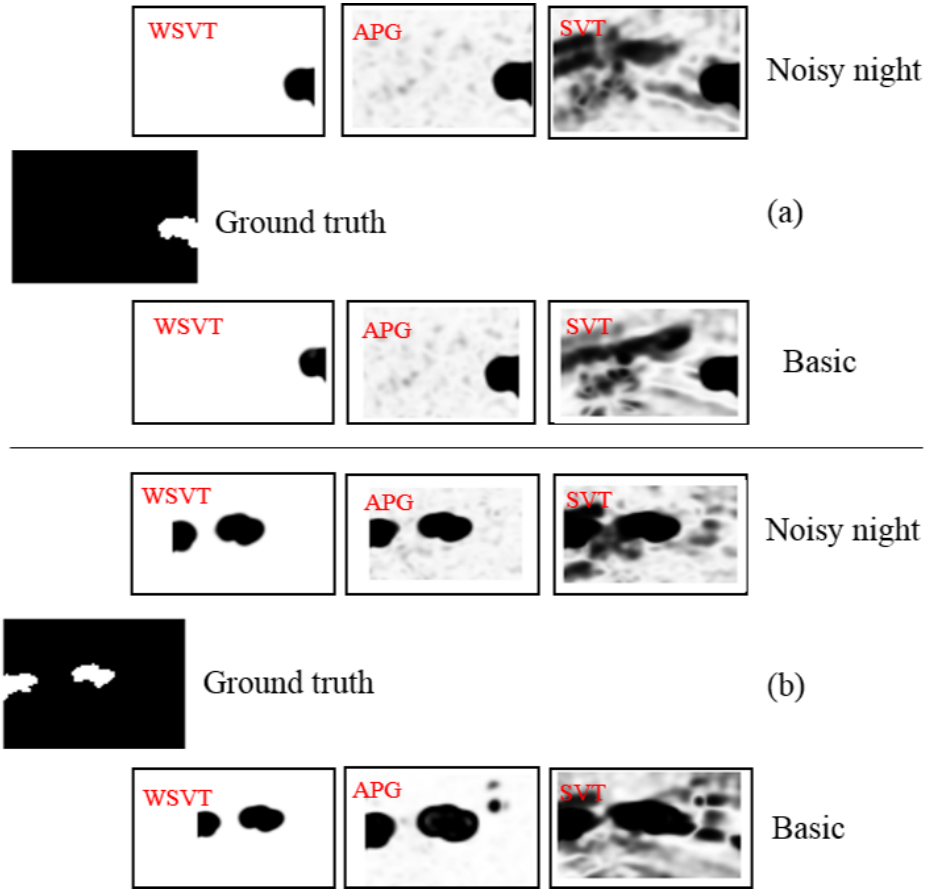}
	\caption{SSIM map of foreground frame:~(a) 600 and (b)~210.~Since iEALM and APG have same recovered foreground we only provide SSIM map for APG.~Indeed WSVT with $\tilde{\lambda}=20$ has the best foreground reconstruction.~The foreground recovered by SVT is very poor and RPCA recovers fragmentary foreground.}\label{ssim_map}
\end{figure}

In Figure~\ref{psnr}, we demonstrate the PSNR and mean PSNR of different methods on the Stuttgart sequence on two different scenarios. For their implementation, we calculate the PSNR of the entire video sequence for each scenario and compare with the ground truth frames. It is evident from Figure~\ref{psnr} that weight improves the PSNR of WSVT significantly over the other existing methods in both scenarios of the Stuttgart sequence. More specifically, we see that in both scenarios the weighted background frames or the frames with least foreground movement have higher PSNR than all other models traditionally used for background estimation.~In Figure~\ref{psnr}~(a), for $\tilde{\lambda}=20$ in WSVT, PSNR of the frames with least foreground movement is above 65 dB.~Similarly, in Figure~\ref{psnr}~(b), for the {\it Noisy night} scenario, the PSNR of the frames with least foreground movement is about 65 dB when $\tilde{\lambda}=20$.

Recently, the structural similarity index~(SSIM) is considered to be one of the most robust quantitative measures and claimed to agree with the human visual perception better compare to the widely used standard measures, such as, MSE and PSNR~\citep{mssim}.~In calculating the MSSIM, we perceive the information how the high-intensity regions of the image are coming through the noise, and consequently, we pay much less attention to the low-intensity regions.~We remove the noisy components from the recovered foreground, $E$, by using the threshold $\epsilon_1$ calculated in Section 4.1.1, such that we set the components below $\epsilon_1$ in $E$ to 0.~In order to calculate the MSSIM of each recovered foreground video frame, we consider an $11\times 11$ Gaussian window with standard deviation~($\sigma$) 1.5 and consider the corresponding ground truth as the reference image. 

The SSIM index of WSVT in Figure~\ref{mssim} shows the superior performance of WSVT especially in presence of the static foreground in both scenarios.~Additionally, from Figure~\ref{mssim} we observe the comparable or superior performance of WSVT on rest of the frames of both scenarios, except a minor deterioration in some frames of the {\it Noisy night} scenario. Moreover, the superior performance of WSVT over other models in Figure~\ref{psnr} and~\ref{mssim} show the validity of the method proposed in Section 4.1.1 in choosing the correct weighted frame indexes.~In Figure~\ref{ssim_map}, we present SSIM index map of two sample foreground video frames of Stuttgart video sequence from both scenarios, which clearly indicate fragmentary foreground recovered by the RPCA algorithms.

\subsection{Facial Shadow Removal}

We also conduct some experiments on the removal of shadow and specularity from face images under varying illuminations and camera positions. The idea was proposed by~\citet{basri}, that the images of the same face exposed to a wide variety of lighting conditions can be approximated accurately by a low-dimensional linear subspace. More specifically, the images under distant, isotropic lighting lie close to a 9-dimensional linear subspace which is known as the {\it harmonic plane}. 	
\begin{figure}[H]
 \includegraphics[width=\textwidth]{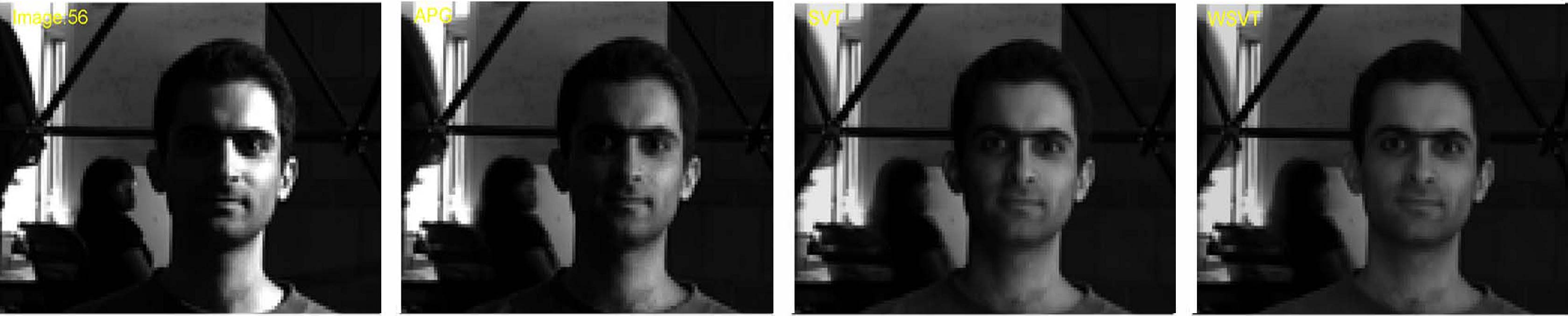}
	\caption{{\bf Qualitative results for facial shadow removal.}~Left to right: Original image~(person B11,~image 56,~partially shadowed) followed by the low-rank approximations using~APG,~SVT,~and WSVT, respectively.~WSVT removes the shadows and specularities uniformly form the face image especially from the left half of the image.}\label{face}
\end{figure}

For facial shadow removal we use test images from the Extended Yale Face Database B~(\citep{yale};~see also, http://vision.ucsd.edu/content/extended-yale-face-database-b-b).~We choose 65 sample images and perform our experiments.~The images are resized to [96,128]; originally they are [480,640]. We set a uniform threshold $10^{-7}$ for each algorithm. For APG and iEALM, $\lambda={1}/{\sqrt{{\rm max}\{m,n\}}}$, and the parameters for iEALM are set to $\mu=1.5, \rho=1.25$~\citep{APG,LinChenMa}. For WSVT we choose $\tau=500,~\mu = 15,\;{\rm and}\;\rho=3$.~The weight matrix is set to $I_n$.~For SVT, we choose $\tilde{\tau}=\tau/\mu$.~Since we have no access to the ground truth for this experiment we will only provide the qualitative result.~Note that the rank of the low-dimensional linear model recovered by RPCA methods is 35, while SVT and WSVT are both able to find a rank 4 subspace.~Figure~\ref{face} shows that WSVT outperforms RPCA algorithms and SVT in terms of the shadow removal results.

\subsection{Computation time comparison}
Tables~\ref{table_1}~and~\ref{table_2} contrast the computation cost of our WSVT to other methods on both sets of experiments: background estimation and facial shadow removal. It is easy to see that WSVT is more efficient than the RPCA algorithms (iEALM and APG) especially when the video sequence is long. Recall that the performance of WSVT is better than or on par with RPCAs. We thus expect that WSVT can be an efficient and effective alternative to RPCAs in more applications.

\begin{table}[H]
	\caption{Average computation time~(in seconds) for each algorithm in background estimation on the Stuttgart sequence {\it Basic} scenario.~All experiments were performed on a computer with 3.1 GHz Intel Core i7 processor and 8GB memory.} \label{table_1}
	\begin{center}
		\begin{tabular}{|c|c|c|c|c|c|}
			\hline No. of frames & iEALM & APG & SVT & WSVT \\
			\hline 200  & 4.994787 & 14.455450  & 0.085675 & 1.4468\\
			\hline 600  & 131.758145 & 76.391438  & 0.307442 & 8.7885334\\
			\hline
		\end{tabular}
	\end{center}
\end{table}
\begin{table}[H]
	\caption{Average computation time~(in seconds) for each algorithm in shadow removal.} \label{table_2}
	\begin{center}
		\begin{tabular}{|c|c|c|c|c|c|}
			\hline No. of images & iEALM & APG & SVT & WSVT \\
			\hline 65  & 1.601427 & 10.221226  & 0.039598 & 1.047922\\
			\hline
		\end{tabular}
	\end{center}
\end{table}
\section{Conclusion and Future Research}

We formulated and studied a weighted version of the SVT method (WSVT) and proposed a numerical algorithm to solve WSVT using augmented Lagrangian function and alternating direction method.~We managed to establish the convergence of our algorithm. Through real data, we demonstrated that by using the weight, which can be learned from the data, we can gain better performance than RPCAs in several applications in computer vision. In particular, our algorithm shows better quantitative results in Stuttgart video sequence and facial shadow removal compare to other state-of-the-art unweighted low-rank algorithms.

Determining the weights adaptively is a challenging and mathematically involved problem and a good starting point for future research. Initially we tuned the parameters for our method using a grid search, but in future the parameters may be trained jointly with the low-rank matrix in a more robust fashion. We also plan on testing the potential applications of our method by exploring research areas like domain adaptation and video summarization.

\section*{Appendix A.}
In this section, we will provide the proofs of the convergence results stated in Section~3. 
\label{app:theorem}
To establish our main results, we need two lemmas. First, we establish the boundedness of $Y_k$: 
\begin{lemma}\label{theorem_3_1}
	The sequence $\{Y_k\}$ is bounded.
\end{lemma}
\begin{proof}~~By the optimality condition for $D_{k+1}$ we have,
\begin{align*}0 \in \partial_DL(C_{k+1},D_{k+1},Y_k,\mu_k).\end{align*}
	So,
	\begin{eqnarray*}
		0 &\in& \tau\partial\|D_{k+1}\|_{\ast} +Y_k +\mu_k(D_{k+1} -C_{k+1}W^{-1}).
	\end{eqnarray*}
	Therefore,
	$-Y_{k+1} \in \tau\partial\|D_{k+1}\|_{\ast}.$
	By using Theorem 4 in~\citep{LinChenMa},~(see also~\citet{watson}), we conclude that the sequence $\{Y_k\}$ is bounded by $\tau$ in the dual norm of $\|\cdot\|_*$. But the dual of $\|\cdot\|_*$ is the spectral norm,~$\|\cdot\|_2$.~So $\|Y_{k+1}\|_2\le\tau$. Hence $\{Y_k\}$ is bounded.
\end{proof}
Next, we show the boundedness of the sequence $\{\hat{Y}_{k}\}$ which requires a different argument.
\begin{lemma}\label{theorem_3_2}~We have the following:\\
		(i) The sequence $\{C_k\}$ is bounded.\\
		(ii) The sequence $\{\hat{Y}_{k}\}$ is bounded.
\end{lemma}
\begin{proof}~~We start with the optimality of $C_{k+1}$:
	\begin{equation*}0 = \frac{\partial}{\partial C} L(C_{k+1},D_{k},Y_k,\mu_k).\end{equation*}
	We get
	\begin{equation}\label{x0}
	(C_{k+1}W^{-1}-X)W W^{T} = Y_k+\mu_k (D_k-C_{k+1}W^{-1}).
	\end{equation}
\begin{enumerate}[(i)]
	\item Solving for $D_k$ in (\ref{x0}), we arrive at
	\begin{equation*}D_k=C_{k+1}(W^{-1}+\frac{1}{\mu_k} W^T)-\frac{1}{\mu_k}(XW W^T-Y_k).\end{equation*}
	Next, using the definition of $\{Y_k\}$ to write
	\begin{equation*}
	D_k=C_k W^{-1}-\frac{1}{\mu_{k-1}} Y_{k-1}+\frac{1}{\mu_{k-1}} Y_k
	\end{equation*}
	and now equating the two expressions for $D_k$ to obtain
	\begin{align*}
	& C_kW^{-1}-\frac{1}{\mu_{k-1}} Y_{k-1}+\frac{1}{\mu_{k-1}} Y_k
	=C_{k+1}(W^{-1}+\frac{1}{\mu_k} W^T)-\frac{1}{\mu_k}(XW W^T-Y_k).\end{align*}
	To simplify the notations, we will use $O(\frac{1}{\mu_k})$ to denote matrices whose norm is bounded by a constant (independent of $k$) times $\frac{1}{\mu_k}$. So, by using the boundedness of $\{Y_k\}$, the above equation can be written as
	\begin{equation}\label{x1}
	C_{k+1}(I+\frac{1}{\mu_k}W^TW)=C_k+O(\frac{1}{\mu_k}).
	\end{equation}
	Diagonalize the positive definite matrix $W^TW$ as $W^TW=Q\Lambda Q^T$ and use it in~(\ref{x1}) to get
	\begin{equation*}
	C_{k+1}Q(I+\frac{1}{\mu_k}\Lambda)=C_kQ+O(\frac{1}{\mu_k}).
	\end{equation*}
	Taking the Frobenius norm on both sides and using the triangle inequality yield
	\begin{equation}\label{x2}
	\|C_{k+1}Q(I+\frac{1}{\mu_k}\Lambda)\|_F\leq \|C_kQ\|_F+O(\frac{1}{\mu_k}).
	\end{equation}
	Since the diagonal matrix $I+\frac{1}{\mu_k}\Lambda$ has all diagonal entries no smaller than $1+\lambda/\mu_k$ where $\lambda>0$ denotes the smallest eigenvalue of $W^TW$, we see that
	\begin{equation*}\|C_{k+1}Q\|_F\leq (1+\frac{\lambda}{\mu_k})^{-1} \|C_{k+1}Q(I+\frac{1}{\mu_k}\Lambda)\|_F.\end{equation*} Thus,  (\ref{x2}) implies
	\begin{equation*}
	\|C_{k+1}Q\|_F\leq (1+\frac{\lambda}{\mu_k})^{-1}\|C_kQ\|_F+O(\frac{1}{\mu_k}),
	\end{equation*}
	which, by the unitary invariance of the norm, is equivalent to
	\begin{equation*}
	\|C_{k+1}\|_F\leq (1+\frac{\lambda}{\mu_k})^{-1}\|C_k\|_F+\frac{C}{\mu_k}~{\rm for~all}~k,
	\end{equation*}
	for some constant $C>0$ independent of  $k$.
	Finally, using the fact that $\mu_{k+1}=\rho\mu_k$ with $\rho>1$, we see that the above inequality implies (by mathematical induction) that $\|C_k\|_F\leq C^*$ for some constant $C^*>0$ (say, $C^*=C(\mu_0+\lambda)/(\mu_0\lambda)$ would work). This completes the proof of the boundedness of $\{C_k\}$.\\
\item[(ii)] Equation (\ref{x0}) gives us $\hat{Y}_{k+1}=(C_{k+1}W^{-1}-X)WW^T$ by using the definition of $\hat{Y}_{k+1}$, and so, the boundedness of $\{\hat{Y}_k\}$ follows immediately from the boundedness of $\{C_k\}$ established in (i) above.
\end{enumerate}
\end{proof}

With the boundedness results of the sequences $\{Y_k\},\{C_k\},$ and $\{\hat{Y}_k\}$, we are ready to prove Theorems~\ref{theorem_3_3} and ~\ref{theorem_3_4}.
\newpage
\begin{proof}{\bf of Theorem~\ref{theorem_3_3}} 
\begin{enumerate}[(i)]
	\item Since $Y_{k+1} - \hat{Y}_{k+1} = \mu_k(D_{k+1} - D_{k})$ we have
\begin{equation*}
D_{k+1} - D_{k} = \frac{1}{\mu_k}(Y_{k+1} - \hat{Y}_{k+1}).
\end{equation*}
So, by the boundedness of $\{Y_k\}$ and $\{\hat{Y}_k\}$ from Lemma~\ref{theorem_3_1} and \ref{theorem_3_2}, for all $k$, we have
\begin{equation*}
\|D_{k+1} - D_{k}\|=O(\frac{1}{\mu_k}),
\end{equation*}
which, by comparison test, implies the convergence of $\{D_k\}$. Now, recall that
\begin{align*}C_{k+1} &= (XW +\mu_kD_k(W^{-1})^T+Y_k(W^{-1})^T)(I +\mu_k(W^TW)^{-1})^{-1}.\end{align*} So, we see that $\{C_k\}$ is convergent as well.
Next,  from the definition of $\{Y_k\}$, we have
\begin{equation*}\frac{1}{\mu_k}(Y_{k+1}-Y_k)=D_{k+1}-C_{k+1}W^{-1}.\end{equation*}
Thus,
\begin{equation}\label{dk}
\|D_{k+1}-C_{k+1}W^{-1}\|=O(\frac{1}{\mu_{k}}).
\end{equation}
\item 
We have, $L_{k+1}=L (C_{k+1},D_{k+1},Y_k,\mu_k) \le L (C_{k+1},D_{k},Y_k,\mu_k)\le L (C_{k},D_{k},Y_k,\mu_k)$. Note that, $L (C_{k},D_{k},Y_k,\mu_k)=L_k+ \frac{\mu_k+\mu_{k-1}}{2}\|D_k-C_kW^{-1}\|_F^2$.~Therefore,
\begin{equation*}L_{k+1}-L_k \le  \frac{\mu_k+\mu_{k-1}}{2}\|D_k-C_kW^{-1}\|_F^2.\end{equation*}
Since $\mu_{k+1}=\rho\mu_k$, we find,~using~(\ref{dk}),
\begin{align*}
L_{k+1}-L_k &\le \frac{\mu_k+\mu_{k-1}}{2}\|D_k-C_kW^{-1}\|_F^2 = O(\frac{1}{\mu_{k}}), {\rm as}\;\;k \rightarrow \infty.
\end{align*}
This completes our proof of Theorem~\ref{theorem_3_3}.
\end{enumerate}
\end{proof}

\begin{proof}{\bf of Theorem~\ref{theorem_3_4}} By Theorem~\ref{theorem_3_3}~(i) and by taking the limit as $k\to \infty$,  we get
\begin{equation}\label{xl1}
C_{\infty}W^{-1}=D_{\infty}.
\end{equation}
Note that
	\begin{align}\label{1st ineq}
	L(C_{k},D_{k},Y_{k-1},\mu_{k-1})&=\min_{C,D}L(C,D,Y_{k-1},\mu_{k-1})\nonumber\\
	&\le \min_{CW^{-1}=D}L(C,D,Y_{k-1},\mu_{k-1})\nonumber\\
	&\le \|XW-C_{\infty}\|_F^2+\tau\|D_{\infty}\|_* \nonumber\\
	&=f_{\infty},
	\end{align}
where we applied (\ref{xl1}) to get the last inequality.	Note also that
	\begin{align*}
	&\|XW-C_{k}\|_F^2+\tau\|D_{k}\|_*\\
	&=L(C_{k},D_{k},Y_{k-1},\mu_{k-1})-\langle Y_{k-1},D_{k}-C_{k}W^{-1}\rangle -\frac{\mu_{k-1}}{2}\|D_{k}-C_{k}W^{-1}\|_F^2,
	\end{align*}
	which, by using the definition of $Y_{k}$ and (\ref{1st ineq}), can be further rewritten into 
	\begin{align}\label{1 st side}
	&\|XW-C_{k}\|_F^2+\tau\|D_{k}\|_*\nonumber\\
	&=L(C_{k},D_{k},Y_{k-1},\mu_{k-1})-\langle Y_{k-1},\frac{1}{\mu_{k-1}}(Y_{k}-Y_{k-1})\rangle -\frac{\mu_{k-1}}{2}\|\frac{1}{\mu_{k-1}}(Y_{k}-Y_{k-1})\|_F^2\nonumber\\
	&\le f_{\infty}+\frac{1}{2\mu_{k-1}}(\|Y_{k-1}\|_F^2-\|Y_{k}\|_F^2).
	\end{align}
	Next, by using triangle inequality we get
	\begin{align}\label{2 nd side}
	&\|XW-C_{k}\|_F^2+\tau\|D_{k}\|_*\nonumber\\
	&\ge\|XW-D_{k}W\|_F^2+\tau\|D_{k}\|_*-\|C_{k}-D_{k}W\|_F^2\nonumber\\
	&\ge f_{\infty}-\|\frac{1}{\mu_{k-1}}(Y_{k-1}-Y_{k})W\|_F^2\nonumber\\
	&=f_{\infty}-\frac{1}{\mu_{k-1}^2}\|(Y_{k-1}-Y_{k})W\|_F^2.
	\end{align}
	Combining~(\ref{1 st side}) and~(\ref{2 nd side}), we obtain the desired result.
	\end{proof}

\bibliography{sample}

\end{document}